\documentclass{article}
\usepackage[utf8]{inputenc}
\usepackage{authblk}

\input{ML.sty}
\input{DR.sty}
\usepackage[margin=1in]{geometry}

\usepackage{color}
\definecolor{ultramarine}{rgb}{0.07, 0.04, 0.56} 
\definecolor{navy}{rgb}{0,0,0.5}

\title{Distributional Robustness and Transfer Learning Through Empirical Bayes}
\author[1]{Michael Law\thanks{Supported in part by NSF Grant DMS-2203012.}}
\author[1]{Peter B\"uhlmann\thanks{Supported in part by European Research Council (ERC), European Union’s Horizon 2020 research and innovation programme, grant agreement No. 786461.}}
\author[2]{Ya\hspace{-.1em}'\hspace{-.1em}acov Ritov\thanks{Supported in part by NSF Grant DMS-2113364.}}
\affil[1]{Seminar for Statistics, ETH Z\"urich}
\affil[2]{Department of Statistics, University of Michigan}
\date{}

\begin{document}

\maketitle

\begin{abstract}
    We consider the problem of statistical inference on parameters of a target population when auxiliary observations are available from related populations.  We propose a flexible empirical Bayes approach that can be applied on top of any asymptotically linear estimator to incorporate information from related populations when constructing confidence regions.  The proposed methodology is valid regardless of whether there are direct observations on the population of interest.  We demonstrate the performance of the empirical Bayes confidence regions on synthetic data as well as on the Trends in International Mathematics and Sciences Study when using the debiased Lasso as the basic algorithm in high-dimensional regression.
\end{abstract}

\section{Introduction}\label{section:intro}

In classical parametric statistics, we are interested in constructing a confidence interval for a parameter $\theta = \theta(\P) \in \Theta$ on the basis of $n$ independent and identically distributed observations from $\P$, which provides a meaningful interpretation for new individuals being drawn from $\P$.  However, by viewing all individuals as coming from a single large population $\P$ might in fact be overly simplistic and ignores heterogeneity within our population.  For example, if $\P = \Pe$ describes the population in environment $k$, we might expect the parameter of interest $\theta(\Pe)$ to vary across our different environments.  Due to this heterogeneity amongst our environments, we may no longer be interested in $\theta(\P)$, the parameter value for the homogenized population, but rather $\theta(\Pe)$, the parameter value for the local population of interest.  A canonical example is the one-way random effects model where the global mean is $\theta(\P) = \mu$ and the local mean is $\theta(\Pe) = \mu + \alpha_{k}$.  Unlike the usual random effect models, we do not view the random component as nuisance, but rather as a source of information that can improve our inference of $\theta(\Pe)$. 

To fix some notation, we write $\Pzero$ to denote the population of interest and $\Pone, \dots, \PK$ to denote the $K$ populations that are similar to $\Pzero$.  Moreover, for $k = 0, 1, \dots, K$, we let $\nk$ denote the number of observations from $\Pk$.  This gives rise to two related, yet distinct, practical problems.  In the first setting, we only have observations from $\Pone, \dots, \PK$, but are interested in a new, typical population $\Pzero$.  This arises naturally in a cluster sampling framework, where we sample $K$ populations but want to generalize to a new population $\Pzero$.  We refer to this setting as ``distributional robustness'' since we want an interval that is robust to small distributional perturbations in our parameter of interest.  In the second setting, we have observations from $\Pone, \dots, \PK$ and some observations from $\Pzero$.  We refer to this as the ``transfer learning'' setting since we try to borrow information from related populations $\Pone, \dots, \PK$ to improve our inference for the target population $\Pzero$.  

For both settings, we consider a unified framework to construct confidence intervals for $\thetazero \defined \theta(\Pzero)$.  The idea is to leverage the heterogeneity across our various populations.  By viewing $\theta = \theta(\Pe)$ 
as a random effect, with the randomness coming from our environments $\e\in\E$, this induces a distribution $\pi$ on the parameter space $\Theta$.  We emphasize that, although $\pi$ is a distribution on the parameter space $\Theta$, our approach is decidedly non-Bayesian in both a methodological and philosophical perspective.  In particular, $\pi$ is not a subjective \emph{a priori} distribution, but rather an objective probability distribution characterizing the variability of $\theta$ induced by our naturally occurring environments.  
Throughout, we consider for simplicity one-dimensional Euclidean parameters (ie., $\Theta \subseteq \Rbb$), but the theory can be extended to multivariate parameters. 

Therefore, we assume the underlying hierarchical model: for every $k = 0, 1, \dots, K$,  
\begin{align}\label{model:parametric:hierarchy}
    \begin{aligned}
        \Xik | \thetak, \etak &\iid f(\cdot | \thetak, \etak), \\
        \thetak &\iid \pi(\cdot)  
    \end{aligned}
\end{align}
for $i = 1, \dots, \nk$.  Here, $\etak$ is the nuisance parameter corresponding to population $\Pk$, which we assume has fixed, deterministic values.  We write $\Xmatk$ to denote the data from population $\Pk$ for $k = 0, \dots, K$, $\Xmatzerominus$ to denote the data from $\Pone, \dots, \PK$, and $\Xmat$ to denote the data from all of the populations.  Likewise, $\thetavec$ denotes the vector of $\thetazero, \dots, \thetaK$.  Then, following \cite{morris1983parametric}, the goal is to construct \emph{empirical Bayes confidence region}.

\begin{definition}\label{definition:ebcr}
    For $0 < \alpha < 1$, a region $\Ical = \Ical(\Xmat) \subseteq \Theta$ is a $1 - \alpha$ \emph{empirical Bayes confidence region for $\thetazero$} if 
    \begin{align*}
        \pr_{\Xmat, \thetavec} \Big( \Ical(\Xmat) \ni \thetazero \Big) \geq 1 - \alpha.
    \end{align*}
\end{definition}

\begin{definition}\label{definition:cr}
    For $0 < \alpha < 1$, a region $\Ical = \Ical(\Xmat) \subseteq \Theta$ is a $1 - \alpha$ \emph{confidence region for $\thetazero$} if 
    \begin{align*}
        \pr_{\Xmat | \thetavec} \Big( \Ical(\Xmat) \ni \thetazero \Big) \geq 1 - \alpha.
    \end{align*}
\end{definition}

The primary difference between empirical Bayes confidence regions and classical confidence regions is unconditional versus conditional coverage of $\thetazero$ over $\thetavec$.  Empirical Bayes confidence regions do not ensure nominal coverage for each fixed value of $\thetazero$, but rather on average over all $\theta \in \Theta$.  In the setting of distributional robustness, where we try to predict the value of a random variable, conditional coverage necessarily implies infinitely large confidence regions if $\Theta$ is unbounded. For transfer learning, depending on the particular application, unconditional coverage may be sufficient.

\subsection{Existing Works and Our Contributions}

Empirical Bayes is one of the earliest tools in statistics introduced by \cite{robbins1956empirical} to combine information from various populations together.  In the context of compound decision theory, empirical Bayes yields estimators that have lower aggregate risk compared to viewing the populations separately.  Though initially used mainly for parameter estimation, \cite{morris1983parametric} proposed using the tools of empirical Bayes for statistical inference; he focused on the parametric problem and does not provide a general methodological approach.  For modern references on empirical Bayes, we refer the interested reader to \cite{Zhang03} and \cite{JiangZhang09} and the references therein. 

The work most similar to ours is that of \cite{ignatiadis2022confidence}, who propose a method to construct confidence intervals for $\Ebb_{\pi}[h(\theta) | \Xmat = \xvec]$ for a known function $h(\cdot)$.  Whereas they try to capture the conditional expectation of a function of $\theta$ given the data, we target the parameter $\theta$ directly.  To help illuminate the difference, consider the special case where $\nzero = 0$.  On one hand, we are interested in predicting the value of $\thetazero = \theta(\Pzero)$ in an unobserved population.  On the other hand, \cite{ignatiadis2022confidence} can only construct a confidence interval for quantities such as $\Ebb_{\pi}[\thetazero]$, the mean in population $\Pzero$ if $h(\cdot)$ is the identity function.  This difference in objectives is reflected in the methodological approaches.  In the terminology of \cite{efron2014two}, the approach of \cite{ignatiadis2022confidence} is that of $f$-modeling, whereby the marginal distribution of the data is estimated.  Conversely, the approach of the present paper is that of $g$-modeling, whereby the distribution of $\pi(\cdot)$ is directly estimated.  

\subsection{Organization}
We end this section with a description of the notation used throughout.  Then, in Section \ref{section:anova}, we further motivate our methodology by revisiting the Gaussian one-way random effects model.  Next, we generalize these ideas to a broader class of models and consider the theoretical performance of an oracle that has access to the true prior in Section \ref{section:ob}.  In Section \ref{section:eb}, we consider general conditions under which the empirical Bayes methodology has asymptotic $1-\alpha$ coverage.  Then, we demonstrate the versatility of our results on asymptotically linear estimators in Section \ref{section:ale} and show how to estimate $\pi(\cdot)$ with a parametric Gaussian prior and a general nonparametric prior in Subsections \ref{subsection:eb:gauss} and \ref{subsection:eb:np}.  Finally, in Sections \ref{section:simulation} and \ref{section:realdata}, we analyze the empirical performance of our proposed methodology on synthetic data and to the Trends in Mathematics and Sciences Study, respectively.  For ease of presentation, we defer the proof of all results and supplemental lemmata to Section \ref{section:appendix}.

\subsection{Notation}
We note that all of our parameters implicitly depend on a parameter $n$; for example, the number of environments is $K = K(n)$ and the sample sizes are $\nzero = \nzero(n), \none = \none(n), \dots, \nK = \nK(n)$.  For ease of presentation, we omit this dependence on $n$ when, by our humble judgement, it should not cause confusion.  Moreover, we define $\thetan$ to be the parameter value $\theta$ of an arbitrary population at $n$ with an asymptotically linear estimator $\thetahatn$, which simultaneously covers low and high-dimensional settings (see Section \ref{section:ale} for a formal definition).  Throughout, we write $\pi(\cdot) = \pi_{n}(\cdot)$ to both denote the density of our random effects over $\Theta$; the dependence of $\pi(\cdot)$ on $n$ allows for contiguity of the prior.  We let $\lambda(\cdot)$ denote Lebesgue measure on $\Rbb^{m}$ for all $m \geq 1$, with the dimension being implicit.  We write $\phi(\cdot | \mu, \sigmasq) = \phi_{\mu, \sigmasq}(\cdot)$ and $\Phi(\cdot | \mu, \sigmasq) = \Phi_{\mu, \sigmasq}(\cdot)$ to denote the density and distribution function of a Gaussian random variable with mean $\mu$ and variance $\sigmasq$.  

\section{Gaussian One-Way Random Effects Model}\label{section:anova}

As a motivating example, consider the simple one-way Gaussian random effects model with covariates:
\begin{align*}
    y_{i,k} = \thetak + \langle \xvec_{i,k}, \betaveck \rangleeuclid + \epsilon_{i,k}
\end{align*}
for $k = 0, \dots, K$ and $i = 1, \dots, \nk$.  For simplicity, we assume that $\pi$ is Gaussian with mean $\mu$ and variance $\sigmapisq$ that are known and $\sigmaepsilonsq$ is also known.  There are a few ways of viewing this as a distributional robustness or transfer learning problem:
\begin{enumerate}
    \item If we have no way to connect the $K$ groups, the best we can do is a linear regression in population $\Pzero$.
    
    \item If $\betavec = \betaveczero = \dots = \betavecK$ and we do not know the prior $\pi$ or $K$ is small, then we can estimate the regression coefficients on the other populations and return this to the adaptive setting.

    \item If $\betaveczero \neq \betaveck$ for $k = 1, \dots, K$ but we know the prior $\pi$, then our only advantage is to use the prior distribution and then consider the posterior.  For simplicity, we only consider a Gaussian $\pi$ with mean $\mu$ and variance $\sigmapisq$, corresponding to the usual Gaussian random effects model.

    \item If $\betavec = \betaveczero = \dots = \betavecK$ and we know the prior $\pi$, then we can combine the two above approaches.
\end{enumerate}

Below, we compute explicitly the variances of the resultant estimators, which implies the width of the resultant interval.  

\begin{enumerate}
    \item In this case, we should use least-squares to estimate $\thetazero$, which is known to be parametrically efficient. 
    Letting $\Xmatzero \in \Rbb^{\nzero \times p}$ denote the design matrix in population $\Pzero$, we have 
    \begin{align*}
        \thetahatzero | \thetazero \sim \Ncal\Big(\thetazero, \frac{\sigmaepsilonsq}{\nzero - \onevec_{\nzero}^\T \Xmatzero (\Xmatzero^\T \Xmatzero)^{-1} \Xmatzero^\T \onevec_{\nzero}} \Big).
    \end{align*}

    \item If $\betavec$ is common, we can estimate all of the parameters jointly on the pooled data.  For simplicity, suppose $K = 1$.  Then, we have
    \begin{align*}
        \begin{pmatrix}
            \yveczero \\ \yvecone
        \end{pmatrix}
        = 
        \begin{pmatrix}
            \onevec_{\nzero} & \zerovec_{\nzero} & \Xmatzero \\ 
            \zerovec_{\none} & \onevec_{\none} & \Xmatone 
        \end{pmatrix} 
        \begin{pmatrix}
            \thetazero \\ \thetaone \\ \betavec 
        \end{pmatrix}
        + 
        \begin{pmatrix}
            \epsilonveczero \\ \epsilonvecone
        \end{pmatrix},
    \end{align*}
    and, hence, 
    \begin{align*}
        \thetahatzero | \thetazero \sim \Ncal\Big( \thetazero, \frac{\sigmaepsilonsq}{\nzero - \onevec_{\nzero}^\T \Xmatzero (\Xmatzero^\T \Xmatzero + \Xmatone^\T \Xmatone)^{-1} \Xmatzero^\T \onevec_{\nzero}} \Big).
    \end{align*}
    More generally, if $K > 1$, we have 
    \begin{align*}
        \thetahatzero | \thetazero \sim \Ncal\Big( \thetazero, \frac{\sigmaepsilonsq}{\nzero - \onevec_{\nzero}^\T \Xmatzero (\sum_{k=0}^{K} \Xmatk^\T \Xmatk)^{-1} \Xmatzero^\T \onevec_{\nzero}} \Big).
    \end{align*}
    Compared to the first case, the variance in the second case is always smaller than or equal to the variance in the first case.  To see this, note that, by the matrix inversion lemma, 
    \begin{align*}
        \Big( \sum_{k=0}^{K} \Xmatk^\T \Xmatk \Big)^{-1}
        = (\Xmatzero^\T \Xmatzero)^{-1} - (\Xmatzero^\T \Xmatzero)^{-1}\Big[ \Big( \sum_{k=1}^{K} \Xmatk^\T \Xmatk \Big)^{-1} + (\Xmatzero^\T \Xmatzero)^{-1} \Big]^{-1} (\Xmatzero^\T \Xmatzero)^{-1}.
    \end{align*}
    Hence, the difference in the two denominators of the variance term is 
    \begin{align*}
        &\Big\{\nzero - \onevec_{\nzero}^\T \Xmatzero (\Xmatzero^\T \Xmatzero)^{-1} \Xmatzero^\T \onevec_{\nzero}\Big\} - \Big\{\nzero - \onevec_{\nzero}^\T \Xmatzero \Big(\sum_{k=0}^{K} \Xmatk^\T \Xmatk\Big)^{-1} \Xmatzero^\T \onevec_{\nzero} \Big\} \\ 
        &\phantleq= \onevec_{\nzero}^\T (\Xmatzero^\T \Xmatzero)^{-1}\Big[ \Big( \sum_{k=1}^{K} \Xmatk^\T \Xmatk \Big)^{-1} + (\Xmatzero^\T \Xmatzero)^{-1} \Big]^{-1} (\Xmatzero^\T \Xmatzero)^{-1} \Xmatzero^\T \onevec_{\nzero} \\
        &\phantleq \geq 0.
    \end{align*}
    The last inequality follows from the positive semidefiniteness of the middle term.

    \item If $\pi$ is Gaussian but the $\betavec$ are different, then we may estimate $\thetahatzero$ and then consider the posterior distribution $\thetazero | \thetahatzero$.  Now, 
    it follows from the first case that the posterior distribution is given by 
    \begin{align*}
        \thetazero | \thetahatzero \sim \Ncal\Big(& \frac{\sigmaepsilonsq}{\sigmapisq\{\nzero - \onevec_{\nzero}^\T \Xmatzero (\Xmatzero^\T \Xmatzero)^{-1} \Xmatzero^\T \onevec_{\nzero}\} + \sigmaepsilonsq} \mu + \frac{\sigmapisq \{\nzero - \onevec_{\nzero}^\T \Xmatzero (\Xmatzero^\T \Xmatzero)^{-1} \Xmatzero^\T \onevec_{\nzero}\}}{\sigmapisq\{\nzero - \onevec_{\nzero}^\T \Xmatzero (\Xmatzero^\T \Xmatzero)^{-1} \Xmatzero^\T \onevec_{\nzero}\} + \sigmaepsilonsq} \thetahatzero, \\
        &\phantleq \frac{\sigmapisq \sigmaepsilonsq}{\sigmapisq \{\nzero - \onevec_{\nzero}^\T \Xmatzero (\Xmatzero^\T \Xmatzero)^{-1} \Xmatzero^\T \onevec_{\nzero}\} + \sigmaepsilonsq}\Big).
    \end{align*}
    In the usual one-way Gaussian random effects model, we assume $\mu = 0$, and so the above reduces to 
    \begin{align*}
        \thetazero | \thetahatzero \sim \Ncal\Big(\frac{\sigmapisq \{\nzero - \onevec_{\nzero}^\T \Xmatzero (\Xmatzero^\T \Xmatzero)^{-1} \Xmatzero^\T \onevec_{\nzero}\}}{\sigmapisq\{\nzero - \onevec_{\nzero}^\T \Xmatzero (\Xmatzero^\T \Xmatzero)^{-1} \Xmatzero^\T \onevec_{\nzero}\} + \sigmaepsilonsq} \thetahatzero, 
        \frac{\sigmapisq \sigmaepsilonsq}{\sigmapisq \{\nzero - \onevec_{\nzero}^\T \Xmatzero (\Xmatzero^\T \Xmatzero)^{-1} \Xmatzero^\T \onevec_{\nzero}\} + \sigmaepsilonsq}\Big).
    \end{align*}
    Note that the variance of the posterior is always smaller than the variance in the first case since 
    \begin{align*}
        &\hspace{-3em}\frac{\sigmaepsilonsq}{\nzero - \onevec_{\nzero}^\T \Xmatzero (\Xmatzero^\T \Xmatzero)^{-1} \Xmatzero^\T \onevec_{\nzero}} - \frac{\sigmapisq \sigmaepsilonsq}{\sigmapisq \{\nzero - \onevec_{\nzero}^\T \Xmatzero (\Xmatzero^\T \Xmatzero)^{-1} \Xmatzero^\T \onevec_{\nzero}\} + \sigmaepsilonsq} \\ 
        &\phantleq = \frac{\sigmaepsilonsq}{\nzero - \onevec_{\nzero}^\T \Xmatzero (\Xmatzero^\T \Xmatzero)^{-1} \Xmatzero^\T \onevec_{\nzero}} \Big( 1 - \frac{\sigmapisq \{\nzero - \onevec_{\nzero}^\T \Xmatzero (\Xmatzero^\T \Xmatzero)^{-1} \Xmatzero^\T \onevec_{\nzero}\}}{\sigmapisq \{\nzero - \onevec_{\nzero}^\T \Xmatzero (\Xmatzero^\T \Xmatzero)^{-1} \Xmatzero^\T \onevec_{\nzero}\} + \sigmaepsilonsq} \Big) \\ 
        &\phantleq > 0.
    \end{align*}
    However, depending on the nature of the design matrices, the variance could be smaller than, equal to, or greater than the second case.  One advantage of this formulation is in the setting where $\nzero = 0$.  In this case, the above reduces to 
    \begin{align*}
        \thetazero \sim \Ncal(0, \sigmapisq),
    \end{align*}
    and so we may still obtain a bounded interval yielding the correct coverage.

    \item When $\betavec$ is shared across all groups and $\pi$ is Gaussian, we may combine the approaches from the second and third case.  Assuming $\mu = 0$, then the posterior distribution of $\thetazero$ 
    \begin{align*}
        \thetazero | \thetahatzero \sim \Ncal\Big(& \frac{\sigmapisq \{\nzero - \onevec_{\nzero}^\T \Xmatzero (\sum_{k=0}^{K} \Xmatk^\T \Xmatk)^{-1} \Xmatzero^\T \onevec_{\nzero}\}}{\sigmapisq\{\nzero - \onevec_{\nzero}^\T \Xmatzero (\sum_{k=0}^{K} \Xmatk^\T \Xmatk)^{-1} \Xmatzero^\T \onevec_{\nzero}\} + \sigmaepsilonsq} \thetahatzero, \\
        &\phantleq \frac{\sigmapisq \sigmaepsilonsq}{\sigmapisq \{\nzero - \onevec_{\nzero}^\T \Xmatzero (\sum_{k=0}^{K} \Xmatk^\T \Xmatk)^{-1} \Xmatzero^\T \onevec_{\nzero}\} + \sigmaepsilonsq} \Big).
    \end{align*}
    By the above calculations, it is immediate that this interval yields the shortest width of the ones considered.
\end{enumerate}

\section{Oracle Bayes}\label{section:ob}

In this section, we generalize the ideas in Section \ref{section:anova} and consider an arbitrary statistic $\Umatzero = \Umatzero(\Xmatzero) \in \Rbb^{\uzero}$.  For example, $\Umatzero$ may be the full data $\Xmatzero$, an estimator $\thetahatzero$ of $\thetazero$, or simply nothing in the case of distributional robustness ($\nzero = \uzero = 0$).  To motivate our general approach, we consider an oracle that has access to $\pi(\cdot)$ and the density of $\Umatzero$, which we denote by $f(\Umatzero | \thetazero)$.  The oracle may directly consider the posterior distribution of $\thetazero | \Umatzero$; in particular, letting $\pi(\cdot | \Umatzero)$ denote the posterior density, we have 
\begin{align*}
    \pi(\thetazero | \Umatzero) = \frac{f(\Umatzero | \thetazero) \pi(\thetazero)}{f(\Umatzero)},
\end{align*}
where 
\begin{align*}
    f(\Umatzero) = \int_{\Theta} f(\Umatzero | \thetazero) \pi(\thetazero) \lambda(d\thetazero)
\end{align*}
is the marginal density of $\Umatzero$.  Then, for some threshold $\tau = \tau_{n}(\alpha) > 0$, we define the oracle Bayes confidence region as 
\begin{align*}
    \Icalob \defined \Big\{ \theta \in \Theta : \pi(\theta | \Umatzero) > \tau \Big\}.
\end{align*}
In the special case of distributional robustness, the posterior density is equal to the prior density and the oracle Bayes confidence region is therefore 
\begin{align*}
    \Icalob = \Big\{ \theta \in \Theta : \pi(\theta) > \tau) \Big\}.
\end{align*}
In both cases, the value of $\tau$ is defined through 
\begin{align*}
    \tau 
    &\defined \argmax_{\tau > 0} \Big\{ \pr_{\Umatzero, \thetazero} (\Icalob \ni \thetazero) \geq 1 - \alpha \Big\} \\ 
    &= \argmax_{\tau > 0} \Big\{ \int_{\Rbb^{\uzero}} f(\Umatzero) \lambda(d\Umatzero) \int_{\Theta} \indic{\vartheta \in \Theta : \pi(\vartheta | \Umatzero) > \tau} \pi(\theta | \Umatzero) \lambda(d\theta) \geq 1- \alpha \Big\}.
\end{align*}
By construction, it follows that $\Icalob$ attains the nominal $1 - \alpha$ coverage.  Note that $\Icalob$ is {almost} mathematically equivalent to the posterior Bayesian highest density region for $\thetazero$ given the data $\Umatzero$.  {Compared to the subjective Bayesian credible region, the critical threshold $\tau$ here also integrates over the marginal distribution of $\Umatzero$, whereas the Bayesian region simply considers the inner integral; more accurately, $\Icalob$ is an oracle marginal Bayesian confidence region, though we omit this distinction in subsequent discussions.  Moreover, we interpret $\Icalob$ as the objective oracle Bayes confidence region for $\thetazero$ as, at the population level over the randomness of our environments, $\thetazero$ is an independent random variable with random effects distribution $\pi(\cdot)$.}  Though the interpretation between $\Icalob$ and Bayesian credible regions are philosophically different, the Bayesian formalism immediately implies desirable theoretical properties for $\Icalob$ as the following proposition demonstrates. 

\begin{proposition}\label{proposition:oracle:length}
    Let $\Ical = \Ical(\Umatzero) \subseteq \Theta$ be an arbitrary random set satisfying 
    \begin{align*}
        \pr_{\Umatzero, \thetazero} (\Ical \ni \thetazero) \geq 1- \alpha.
    \end{align*}
    If 
    \begin{align*}
        \pr_{\Umatzero, \thetazero} (\Icalob \ni \thetazero) = 1- \alpha,
    \end{align*}
    then 
    \begin{align*}
        \Ebb_{\Umatzero, \thetazero} \lambda(\Icalob) \leq \Ebb_{\Umatzero, \thetazero} \lambda(\Ical).
    \end{align*}
\end{proposition}

Proposition \ref{proposition:oracle:length} asserts that the oracle Bayes confidence region has the smallest expected Lebesgue measure amongst all random sets covering $\thetazero$ with probability at least $1 - \alpha$ over the joint distribution of $\Umatzero$ and $\thetazero$.  The assumption that the oracle Bayes region has exact $1-\alpha$ coverage is a technical requirement; it is a consequence of the mapping $\alpha \mapsto \tau(\alpha)$ not being necessarily injective.  As an example, if $\pi(\cdot)$ is the uniform density on $(0,1)$ and $\uzero = 0$, then $\tau(\alpha) = 1$ for all $\alpha \in (0,1)$.  This is analogous to the non-randomized likelihood ratio test not being the uniformly most powerful test in a simple vs simple hypotheses testing problem when size cannot be attained without randomization.  By introducing additional, external randomization, the oracle Bayes confidence region defined above may be modified such that the coverage is exactly $1 - \alpha$ and, hence, have the smallest expected Lebesgue measure.

\begin{remark}
    If $\Umatzero$ is a sufficient statistic for $\thetazero$, then the conclusion may be strengthened to the oracle Bayes confidence region has the smallest expected Lebesgue measure amongst all random sets covering $\thetazero$ with probability at least $1 - \alpha$ over the joint distribution of $\Xmatzero$ and $\thetazero$, which follows immediately from the factorization theorem.  Thus, by conditioning on the full data $\Xmatzero$, we have the shortest possible expected Lebesgue measure.
\end{remark}

\section{Empirical Bayes}\label{section:eb}

In the previous section, we considered the performance of the oracle Bayes confidence region when the prior $\pi(\cdot)$ and the density $f(\cdot | \thetazero)$ are known exactly.  However, such oracles are not particularly abundant and, in practice, we need to estimate the prior distribution $\pi$.  To this end, let $\pihatn(\cdot)$ be an estimator of $\pin(\cdot)$ and $\fhatn(\cdot | \theta)$ be an estimator of $\fn(\cdot | \thetazero)$, the conditional density of $\Umatzero | \thetazero$.  Here, we include the dependence on $n$ to emphasize that the subsequent limits are taken with respect to $n$ and to allow for contiguity.  Then, the estimated posterior distribution of $\thetazero | \Umatzero$ is given by 
\begin{align*}
    \pihatn(\thetazero | \Umatzero) = \frac{\fn(\Umatzero | \thetazero) \pihatn(\thetazero)}{\fhatn(\Umatzero)},
\end{align*}
where 
\begin{align*}
    \fhatn(\Umatzero) = \int_{\Theta} \fn(\Umatzero | \thetazero) \pihatn(\thetazero) \lambda(d\thetazero)
\end{align*}
is the estimated marginal density of $\Umatzero$.  From here, the natural analogue of the oracle Bayes confidence region is given by the following empirical Bayes confidence region:
\begin{align}\label{equation:definition:eb}
    \Icaleb \defined \Big\{ \theta \in \Theta : \pihat(\theta | \Umatzero) > \tauhat \Big\},
\end{align}
where
\begin{align}\label{equation:definition:eb:tau}
    \tauhat = \tauhat_{n}(\alpha) = \argmax_{\tau > 0} \Big\{ \int_{\Rbb^{\uzero}} \fhatn(\Umatzero) \lambda(d\Umatzero) \int_{\Theta} \indic{\vartheta \in \Theta : \pihatn(\vartheta | \Umatzero) > \tau} \pihatn(\theta | \Umatzero) \lambda(d\theta) \geq 1 - \alpha \Big\}.
\end{align}

To analyze the performance of $\Icaleb$, we impose the following mild assumption.

\begin{assumption}\label{assumption:consistent}
    There exist estimators $\pihatn(\cdot)$ of $\pin(\cdot)$ and $\fhatn(\cdot)$ of $\fn(\cdot)$ such that 
    \begin{align*}
        \Vert \pihatn(\cdot) - \pin(\cdot) \Vertone = \op(1)
    \end{align*}
    and 
    \begin{align*}
        \sup_{\thetazero \in \Upsilon} \Vert \fhatn(\cdot | \thetazero) - \fn(\cdot | \thetazero) \Vertone = \op(1)
    \end{align*}
    for all compacts $\Upsilon \subseteq \Theta$.  Moreover, $\pin$ converges in total variation distance to a distribution $\piinfty$.  
\end{assumption}

Assumption \ref{assumption:consistent} requires consistent estimators of both $\pin(\cdot)$ and $\fn(\cdot | \theta)$ in the $L^{1}$ norm; we defer a more detailed discussion on how to estimate $\fn(\cdot | \theta)$ to Section \ref{section:ale} and $\pin(\cdot)$ to Subsections \ref{subsection:eb:gauss} and \ref{subsection:eb:np} below.  Under the above assumption, we have the following proposition regarding the empirical Bayes confidence region.

\begin{proposition}\label{proposition:totalvariation}
    Suppose Assumption \ref{assumption:consistent} holds.  Then, 
    \begin{align*}
        \int_{\Theta \times \Rbb^{\uzero}} \Big| \fn(\Umatzero) \pin(\thetazero | \Umatzero) - \fhatn(\Umatzero) \pihatn(\thetazero | \Umatzero) \Big| \lambda(d\thetazero \times d\Umatzero) = \op(1).
    \end{align*}
\end{proposition}

Proposition \ref{proposition:totalvariation} asserts that the estimated joint density of $(\Umatzero, \thetazero)$ converges to the population joint density.  By Scheff\'e's Lemma, this implies that empirical and population measures converge in total variation distance, which immediately yields the following corollary regarding the asymptotic coverage.

\begin{corollaryproposition}\label{corollary:coverage}
    Suppose Assumption \ref{assumption:consistent} holds.  Then, 
    \begin{align*}
        \liminf_{n\to\infty} \pr_{\Umatzero, \thetazero} (\Icaleb \ni \thetazero) \geq 1 - \alpha.
    \end{align*}
\end{corollaryproposition}

From here, we immediately have that the empirical Bayes confidence region has asymptotic $1 - \alpha$ coverage as long as $\pihatn(\cdot)$ and $\fhatn(\cdot | \thetazero)$ converge in total variation.  The following example revisits the Gaussian one-way random effects model when $\betaveczero$, $\sigmapisq$, and $\sigmaepsilonsq$ need to be estimated from the data.

\begin{example}
    Consider setting 3 from Section \ref{section:anova}.  That is, we have observations from the one-way Gaussian random effects model 
    \begin{align*}
        y_{i,k} = \thetak + \langle \xvec_{i,k}, \betaveck \rangleeuclid + \epsilon_{i,k}
    \end{align*}
    for $k = 0, \dots, K$ and $i = 1, \dots, \nk$.  Here, $\thetak \iid \Ncal(0, \sigmapisq)$ and $\epsilon_{i,k} \iid \Ncal(0, \sigmaepsilonsq)$.  In matrix notation, we write
    \begin{align*}
        \yveck = \Xmatk \betaveck + \thetak \onevec_{\nk} + \epsilonveck.
    \end{align*}
    Then, for $k = 0, \dots, K$, we let $\thetahatk$ denote the least-squares estimator of $\thetak$ in population $\Pk$ when regressed separately.  Writing $\Zmatk = (\onevec_{\nk}, \Xmatk) \in \Rbb^{\nk \times (p + 1)}$ to denote the combined design matrix, we have that 
    \begin{align*}
        [\identity_{\nk} - \Zmatk (\Zmatk^\T \Zmatk)^{-1} \Zmatk^\T] \yveck \sim \Ncal \Big(0, \sigmaepsilonsq [\identity_{\nk} - \Zmatk (\Zmatk^\T \Zmatk)^{-1} \Zmatk^\T]\Big)
    \end{align*}
    and 
    \begin{align*}
        \thetahatk \sim \Ncal \Big(0, \sigmapisq + \frac{\sigmaepsilonsq}{\nk - \onevec_{\nk} \Xmatk (\Xmatk^\T \Xmatk)^{-1} \Xmatk^\T \onevec_{\nk}}\Big).
    \end{align*}
    Therefore, to estimate $\sigmapisq$ and $\sigmaepsilonsq$, we consider 
    \begin{align*}
        \sigmaepsilonhatsq \defined \frac{\sum_{k=1}^{K} \Vert [\identity_{\nk} - \Zmatk (\Zmatk^\T \Zmatk)^{-1} \Zmatk^\T] \yveck \Verttwo^{2}}{\sum_{k=1}^{K} \nk - Kp}
    \end{align*}
    and 
    \begin{align*}
        \sigmapihatsq \defined \frac{1}{K} \sum_{k=1}^{K} \Big( \thetahatk^{2} - \frac{\sigmaepsilonhatsq}{\nk - \onevec_{\nzero} \Xmatk (\Xmatk^\T \Xmatk)^{-1} \Xmatk^\T \onevec_{\nzero}} \Big).
    \end{align*}
    Supposing $\sigmapihatsq$ and $\sigmaepsilonhatsq$ satisfy
    \begin{enumerate}
        \item $\sigmapihatsq - \sigmapisq = \op(\sigmapisq)$; and 
        \item $\sigmaepsilonhatsq - \sigmaepsilonsq = \op(\sigmaepsilonsq)$,
    \end{enumerate}
    the straightforward calculations imply that the empirical Bayes confidence interval is given by 
    \begin{align*}
        \Icaleb 
        &= 
        \Big( \frac{\sigmapihatsq }{\sigmapihatsq + \ssq} \thetahatzero \pm \zalpha \sqrt{\frac{\sigmapihatsq \ssq}{\sigmapihatsq + \ssq}} \Big),
    \end{align*}
    where 
    \begin{align*}
        \ssq \defined \frac{\sigmaepsilonhatsq}{\nzero - \onevec_{\nzero} \Xmatzero (\sum_{k=0}^{K} \Xmatk^\T \Xmatk)^{-1} \Xmatzero^\T \onevec_{\nzero}}.
    \end{align*}
    To show asymptotic validity, we verify Assumption \ref{assumption:consistent} holds.  Indeed, for an intermediate point $\xi^{2}$ between $\sigmapisq$ and $\sigmapihatsq$, we have
    \begin{align*}
        \Vert \pihat(\cdot) - \pi(\cdot) \Vertone &=\int_{-\infty}^{\infty} \Big| \frac{1}{\sqrt{2\pi \sigmapihatsq}} \exp\Big( -\frac{\thetazero^{2}}{2\sigmapihatsq}\Big) - \frac{1}{\sqrt{2\pi \sigmapisq}} \exp\Big( -\frac{\thetazero^{2}}{2\sigmapisq}\Big) \Big| \lambda(d\theta) \\ 
        & =
        \int_{-\infty}^{\infty} \Big| \frac{\thetazero^{2}}{2\xi^{4}} - \frac{1}{2\xi^{2}} \Big| (\sigmapihatsq - \sigmapisq) \frac{1}{\sqrt{2\pi \sigmapisq}} \exp\Big( -\frac{\thetazero^{2}}{2\sigmapisq}\Big) \lambda(d\theta) \\ 
        &\leq \Big| \frac{\sigmapisq(\sigmapihatsq - \sigmapisq)}{2\xi^{4}} \Big| + \Big|\frac{\sigmapihatsq - \sigmapisq}{2\xi^{2}} \Big| \\ 
        &= \op(1).
    \end{align*}
    By the exact same calculation, it follows that 
    \begin{align*}
        \sup_{\thetazero \in \Rbb} \Vert \fhatn(\cdot | \thetazero) - \fn(\cdot | \thetazero) \Vertone = \op(1).
    \end{align*}
    Thus, we conclude that 
    \begin{align*}
        \liminf_{n \to \infty} \pr_{\thetahatzero, \thetazero} (\Icaleb \ni \thetazero) \geq 1 - \alpha.
    \end{align*}
    Note that the requirements on $\sigmapihatsq$ and $\sigmaepsilonhatsq$ are very mild.  In the case where $\sigmapisq$ and $\sigmaepsilonsq$ are constant, we only require $K \to \infty$.  However, our calculations above also apply in the contiguous setting where $\nzero \sigmapisq \asymp \sigmaepsilonsq$ assuming that $K \to \infty$ sufficiently fast.
\end{example}

So far, we have only considered the asymptotic validity of our empirical Bayes confidence interval, though this raises the question of optimality; in particular, does the ratio of the Lebesgue measure between the empirical Bayes confidence region and the oracle Bayes confidence region converges to one?  As the following example demonstrates, this is not true without further assumptions. 

\begin{example}
    Consider the setting where $\nzero = \uzero = 0$ and $\pin(\cdot)$ is the uniform density on $(0,1)$.  Suppose 
    \begin{align*}
        \pihatn(\theta) = 
        \begin{cases}
            1 + \frac{1}{n} & \text{ if } \theta \in (0, 1-\alpha); \\
            1 + \frac{1}{n} - \frac{1}{\alpha n} & \text{ if } \theta \in (1-\alpha, 1)
        \end{cases}
    \end{align*}
    for $n \geq 1/\alpha - 1$.  Then, it is clear that $\Vert \pihatn - \pin \Vertone \to 0$ since 
    \begin{align*}
        \int_{0}^{1} |\pihatn(\thetazero) - \pin(\thetazero)| \lambda(d\thetazero) = \frac{1-\alpha}{n} + \frac{1 - \alpha}{\alpha^{2} n} \to 0.
    \end{align*}
    However, we have $\Icalob = (0,1)$ and $\Icaleb = (0, 1-\alpha)$ for all $n \geq 1/\alpha - 1$.  Hence, it follows that 
    \begin{align*}
        \frac{\lambda(\Icalob)}{\lambda(\Icaleb)} = \frac{1}{1-\alpha}.
    \end{align*}
\end{example}

\section{Asymptotically Linear Estimators}\label{section:ale}

In this section, we consider the setting where the statistic $\Umatzero$ is an asymptotically linear estimator $\thetahatzero$ of $\thetazero$.  We note that this is a very general setting as most classical confidence intervals are constructed through an asymptotic Gaussian pivot, which usually are asymptotically linear.  Following the notation of \cite{bickel1993efficient}, for $k = 0, 1, \dots, K$, we assume that 
\begin{align*}
    \thetatildek = \thetak + \frac{1}{\nk} \sum_{i=1}^{\nk} \psi(X_{i,k}; \thetak, \etak) + \deltak.
\end{align*}
Here, $\psi(\cdot ; \thetak, \etak)$ is the influence function, $\etak \in \H \subseteq \Rbb$ is a nuisance parameter, and $\deltak = \op(\nk^{-1/2})$ is a small remainder term.  As we formalize in Assumption \ref{assumption:prior:tv} below, we assume that $\psi(\cdot ; \thetak, \etak)$ is asymptotically independent for $\thetak$; in particular, the asymptotic variance does not depend on $\thetak$.

For ease of presentation, we start by deriving results for a generic asymptotically linear estimator of $\thetan \sim \pin$ based on $n$ independent and identically distributed observations $X_{i} \sim \P_{\thetan}$ given by
\begin{align*}
    \thetatilden = \thetan + \frac{1}{n} \sum_{i=1}^{n} \psi(X_{i}; \thetan, \eta) + \deltan.
\end{align*}
In general, the distribution of $\thetatilden | \thetan$ is intractable or difficult to compute.  However, due to the asymptotic linearity, assuming that $\psi(\cdot; \thetan, \eta)$ has a second moment, we have from the central limit theorem that there exists a variance $\sigmasq$ such that 
\begin{align*}
    \sqrt{n} (\thetatilden - \thetan) | \thetan \cond \Ncal(0, \sigmasq).
\end{align*}
Then, rather than directly using the distribution of $\thetatilden$, we may wish to approximate it with its asymptotic distribution.  We emphasize that naively substituting the Gaussian distribution in place of the true distribution does not satisfy Assumption \ref{assumption:consistent}.

To see this, we recall a result of \cite{prokhorov1952local} (see also \cite{rao1960limit}), which asserts that the central limit theorem holds in total variation distance if and only if the distribution of the partial sum of the first $n$ terms is non-singular for some $n$. 
If the distribution of $\sum_{i=1}^{n} \psi(X_{i}; \thetan, \eta)$ is always singular and $\deltan = 0$ almost surely --- for example, the sample mean of a discrete distribution --- then the total variation distance between $\thetatilden | \thetan$ and the Gaussian distribution is always one.  In this case, the asymptotic coverage arising from these posterior level sets does not attain the nominal $1 - \alpha$ level.

Thus, for technical reasons, we have to slightly modify our estimator to be 
\begin{align*}
    \thetahatn \defined \thetatilden + \frac{1}{\sqrt{n}} \xin,
\end{align*}
where $\xin \sim \Ncal(0, \varsigmasqn)$ for some sequence of positive constants $\{\varsigmasqn\}_{n=1}^{\infty}$ decreasing to zero slowly (see Proposition \ref{proposition:fn} below). 
Even if the distribution of $\psi(X_{i}; \thetan, \eta)$ is singular, by adding independent Gaussian noise, the infinite divisibility property of the Gaussian distribution implies that the above display is again a partial sum of independent and identically distributed random variables, but now with a non-singular component.  Therefore, the addition of $\xin / \sqrt{n}$ produces a smoothing effect on the resultant estimator to allow for convergence in total variation.  However, the sequence $\{\varsigmasqn\}_{n=1}^{\infty}$ cannot tend to zero too quickly to ensure a sufficient amount of smoothing.

We start with a simple sufficient condition to approximate the density of $\sqrt{n} \thetahatn | \thetan$.  

\begin{assumption}\label{assumption:berryesseen}
    The influence function satisfies
    \begin{align*}
        \sup_{\theta \in \Upsilon} \Ebb [\psi^{3}(X_{i} ; \theta, \eta)] < \infty
    \end{align*}
    for all compacts $\Upsilon \subseteq \Theta$.  Moreover, the densities $\{\fn(\cdot | \theta)\}_{n=1}^{\infty}$ are bounded and continuous functions of both arguments.  
\end{assumption}

Assumption \ref{assumption:berryesseen} is a mild assumption to ensure uniform convergence of the central limit theorem over compact subsets $\Upsilon\subseteq \Theta$.  If the third moment is finite and continuous as a function of $\theta$, then the above condition holds.  

\begin{example}[Ordinary least-sqares and debiased lasso]\label{example:linearmodel} 
    As an example of an estimator satisfying Assumption \ref{assumption:berryesseen}, we consider the linear model 
    \begin{align*}
        \yveck = \Xmatk \betaveck + \epsilonveck.  
    \end{align*}
    where $\yveck, \epsilonveck \in \Rbb^{\nk}$, $\Xmat \in \Rbb^{\nk \times p}$, and $\betaveck \in \Rbb^{p}$.  Moreover, our parameter of interest is $\thetak = \betavec_{k, 1}$, the first entry of $\betaveck$.  Then, in the low-dimensional setting, the least-squares estimator is both asymptotically linear with 
    \begin{align*}
        \betavechatk = \betaveck + \frac{1}{\nk} \sum_{i=1}^{\nk} (\Ebb \Xmatk^\T \Xmatk)^{-1} \xvec_{k, i} \epsilon_{k, i} + \op(\nk^{-1});
    \end{align*}
    for example, see page 35 of \cite{bickel1993efficient}.  Similarly, \cite{van2014asymptotically} show that the debiased lasso has the same influence function as the least-squares estimator restricted to the sub-model spanned by the active covariates.  In both cases, as long as the distribution of $\epsilon$ admits a finite third moment, Assumption \ref{assumption:berryesseen} is satisfied.  Moreover, if $\epsilonveck$ has a distribution that is absolutely continuous with respect to Lebesgue measure, such as the Gaussian distribution, then no smoothing is required for the least-squares estimator since $\betavechatk$ is non-singular.
\end{example}

Letting $\fn(\cdot | \thetan)$ denote the density of $\sqrt{n} \thetahatn | \thetan$, the following proposition shows that $\fn(\cdot | \thetan)$ converges in total variation distance to $\phi(\cdot | \sqrt{n} \thetan, \sigmasq)$.

\begin{proposition}\label{proposition:fn}
    If Assumption \ref{assumption:berryesseen} holds, then there exists a smoothing sequence $\{\varsigmasqn\}_{n=1}^{\infty}$ such that
    \begin{align*}
        \sup_{\theta \in \Upsilon} \Vert \fn(\cdot | \thetan) - \phi(\cdot | \sqrt{n}\thetan, \sigmasq) \Vertone = o(1)
    \end{align*}
    for all compacts $\Upsilon \subseteq \Theta$.  
\end{proposition}

Proposition \ref{proposition:fn} shows that, for a suitable regularizing sequence $\{\varsigmasqn\}_{n=1}^{\infty}$, the distribution of $\thetahatn | \thetan$ can be well approximated by a Gaussian distribution in the sense of Assumption \ref{assumption:consistent}.  Since $\sigmasq$ is unknown, it must be estimated from the data; however, in general, a consistent estimator of $\sigmasq$ is used to construct classical confidence intervals based on the asymptotic distribution of $\thetatilden$.  In this case, we have the following corollary.

\begin{corollaryproposition}\label{corollary:fn}
    Under the setting of Proposition \ref{proposition:fn}, if $\sigmahatsq / \sigmasq \conp 1$, then 
    \begin{align*}
        \sup_{\theta \in \Upsilon} \Vert \fn(\cdot | \thetan) - \phi(\cdot | \sqrt{n} \thetan, \sigmahatsq) \Vertone = \op(1).
    \end{align*}
\end{corollaryproposition}

Now that we have found a suitable approximation of $\fn(\cdot | \thetan)$, it is left to approximate the distribution $\pin(\cdot)$.  We may again leverage the asymptotic linearity structure of $\thetahatn$.  If $\psi(\cdot ; \thetan, \eta)$ is asymptotically independent of $\thetan$, then $\thetahatn$ is approximately $\thetan$ plus an independent Gaussian term, implying that the marginal distribution of $\thetahatn$ is close to $\pin$ convolved with the Gaussian distribution.  Before formalizing this intuition, for a scaling sequence $\an$, we let $\nun$ denote the distribution of $\an(\thetan - \Ebb \thetan)$ and make the following assumption.
\begin{assumption}\label{assumption:prior:tv}
    The density $\nun(\cdot)$ converges in $\Ltwo$ to a continuous density $\nuinfty(\cdot)$.  Moreover, $\psi(\cdot ; \thetan, \eta)$ is asymptotically independent of $\thetan$.  
\end{assumption}
Letting $\qn(\cdot)$ be the marginal density of $\an (\thetahatn - \Ebb \thetan)$, we have the following proposition.

\begin{proposition}\label{proposition:qconvergence}
    Suppose Assumption \ref{assumption:prior:tv} holds.  
    \begin{enumerate}
        \item If $\an = o(\sqrt{n})$, then there exists a smoothing sequence $\{\varsigmasqn\}_{n=1}^{\infty}$ such that 
        \begin{align*}
            \Vert \qn - \nuinfty \Verttwo \to 0.
        \end{align*}
        
        \item If $\an = \sqrt{n}$, then there exists a smoothing sequence $\{\varsigmasqn\}_{n=1}^{\infty}$ such that 
        \begin{align*}
            \Vert \qn - \nuinfty \ast \phi_{0, \sigmasq} \Verttwo \to 0.
        \end{align*}
    \end{enumerate}
    In both cases, the convergence also holds in $\Lone$.  
\end{proposition}

Proposition \ref{proposition:qconvergence} asserts that the marginal density of $\thetahatn$ is close to the prior density $\pin$ or the convolution of $\pin$ with a Gaussian density, depending on the signal-to-noise ratio regime.  In the subsequent two subsections, we leverage this result to construct estimators of $\pin(\cdot)$ depending on whether $\pin(\cdot)$ is assumed to be Gaussian or not.  

\subsection{Gaussian Empirical Bayes}\label{subsection:eb:gauss}

In this subsection, we generalize the Gaussian one-way random effects model from Section \ref{section:anova} and assume that $\pin(\cdot) = \pin(\cdot | \mupi, \sigmapisq)$ is a Gaussian distribution with mean $\mupi$ and variance $\sigmapisq$.  Here, $\sigmapisq$ may change with $n$, enabling a contiguous prior, but we omit this dependence for simplicity.  From Proposition \ref{proposition:qconvergence}, we have that $\thetahatn$ has an approximate marginal distribution of $\Ncal(\mupi, \sigmapisq + \sigmasq / n)$.  Since the $K$ populations are independent, we may consider the maximum likelihood estimators given by
\begin{align*}
    (\mupitilde, \sigmapitildesq) \defined \argmin_{(\mupi, \sigmapisq) \in \Rbb \times \Rbb_{+}} \Big\{ \prod_{k=1}^{K} \phi(\thetahatk | \mupi, \sigmapisq + \sigmahatsq / \nk) \Big\}.
\end{align*}
which reduces to 
\begin{align*}
    \mupitilde \defined \frac{1}{K} \sum_{k=1}^{K} \thetahatk 
    \und
    \sigmapitildesq \defined \max\Big\{ \frac{1}{K} \sum_{k=1}^{K} (\thetahatk - \mupihat)^{2} - \frac{\sigmahatsq}{m}, 0 \Big\}
\end{align*}
when $m = \none = \dots = \nK$.  However, if $\sigmapitildesq = 0$, then the resultant posterior is a point mass and the empirical Bayes confidence region is the singleton $\mupitilde$ regardless of $\thetahatzero$.  To avoid this problem, we take a regularizing sequence $\{\zetan^{2}\}_{n=1}^{\infty}$ satisfying $\zetan^{2} = o(\sigmapisq)$ and consider 
\begin{align*}
    \mupihat \defined \mupitilde \und \sigmapihatsq \defined \max\{\sigmapitildesq, \zetan^{2}\}.
\end{align*}

In both cases, we estimate $\pin(\cdot)$ by $\pihatn(\cdot) = \phi(\cdot | \mupihat, \sigmapihatsq)$, leading to the estimated empirical Bayes confidence interval as 
\begin{align}\label{equation:peb}
    \begin{aligned}
        \Icaleb \defined 
        \Big( 
        \frac{\nzero \sigmapihatsq}{\sigmahatsq + \nzero \sigmapihatsq} \thetahatzero + \frac{\sigmahatsq}{\sigmahatsq + \nzero \sigmapihatsq} \mupihat \pm \zalpha \sqrt{\frac{\sigmapihatsq \sigmahatsq}{\sigmahatsq + \nzero \sigmapihatsq}}
        \Big).
    \end{aligned}
\end{align}
This interval should be compared with the classical large sample confidence interval arising from using $\thetatildezero$ as an asymptotic pivot, 
\begin{align*}
     \Icalcl \defined \Big( \thetatildezero \pm \zalpha \sqrt{\frac{\sigmahatsq}{\nzero}} \Big).
\end{align*}
Under the following mild assumption, the above interval in equation \eqref{equation:peb} attains asymptotic coverage of the parameter $\thetazero$ and has strictly smaller Lebesgue measure than the classical large sample interval.  

\begin{assumption}\label{assumption:peb:consistent}
    The following three conditions hold:
    \begin{enumerate}
        \item $m, K \to \infty$; 

        \item $\max_{k=1, \dots, K} |\deltak| = \op(\sigmapi)$; and 

        \item $\max\{\sigmahatsq - \sigmasq, \sigmasq\} = \op(m\sigmapisq)$.
    \end{enumerate}
\end{assumption}

\begin{theorem}\label{theorem:parametric}
    Suppose Assumptions \ref{assumption:berryesseen}, \ref{assumption:prior:tv}, and \ref{assumption:peb:consistent} hold.  If $\pin(\cdot) = \phi(\cdot | \mupi, \sigmapisq)$, then
    \begin{align*}
        \Vert \phi_{\mupi, \sigmapisq} - \phi_{\mupihat, \sigmapihatsq} \Vertone = \op(1).
    \end{align*}
    Moreover, the interval defined in equation \eqref{equation:peb} satisfies
    \begin{align*}
        \liminf_{n \to \infty} \pr_{\thetahatzero, \thetazero} (\Icaleb \ni \thetazero) \geq 1 - \alpha
    \end{align*}
    with 
    \begin{align*}
        \frac{\lambda(\Icaleb)}{\lambda(\Icalcl)} = 1 - \sqrt{\frac{\nzero \sigmapihatsq}{\sigmahatsq + \nzero \sigmapihatsq}}.
    \end{align*}
\end{theorem}

\subsection{Nonparametric Empirical Bayes}\label{subsection:eb:np}
In this subsection, we consider the more general problem when $\pin(\cdot)$ is a general nonparametric prior (cf. \cite{BrownGreenshtein09}).  From Proposition \ref{proposition:qconvergence}, there are two distinct settings depending on the scaling rate $\an$.  When $\an = o(\sqrt{n})$, the noise of $\thetahatn$ to $\thetan$ is asymptotically negligible relative to the strength of the signal.  Thus, in this case, for a bandwidth parameter $\{\bn\}_{n=1}^{\infty}$ decreasing to zero, let
\begin{align*}
    \pihatn(t) = \frac{1}{K} \sum_{k=1}^{K} \phi(t | \thetahatk, \bnsq).
\end{align*}

In this case, the posterior distribution of $\thetazero | \thetahatzero$ can be computed explicitly as the estimated prior is a Gaussian mixture.  For $k = 1, \dots, K$, let
\begin{align*}
    \gammak(\thetahatzero) \defined \frac{\phi(\thetahatzero | \thetahatk, \bnsq)}{\sum_{j=1}^{K} \phi(\thetahatzero | \thetahat_{j}, \bnsq)}.
\end{align*}
Then, the posterior density of $\thetazero | \thetahatzero$ is
\begin{align*}
    \pihatn(x | \thetahatzero) = \sum_{k=1}^{K} \gammak(\thetahatzero) \phi\Big( x \Big| \frac{\sigmahatsq}{\nzero \bnsq + \sigmahatsq} \thetahatk + \frac{\nzero \bnsq}{\nzero \bnsq + \sigmahatsq} \thetahatzero, \frac{\bnsq \sigmahatsq}{\nzero \bnsq + \sigmahatsq} \Big).
\end{align*}

On the other hand, when $\an = \sqrt{n}$, the signal and the noise are of comparable magnitude and, hence, the contribution due to the noise needs to be removed.  Therefore, following \cite{meister2009deconvolution}, we consider a deconvolution estimator
\begin{align*}
    \pitilden(x) = 
    \frac{1}{2\pi K} \sum_{k=1}^{K} \int_{-1/\bn}^{1/\bn} \exp( -\ifrak x z) \exp( \ifrak z \thetahatk) \exp \Big( \frac{\sigmahatsq z^{2}}{2m} \Big) \lambda(dz)
\end{align*}
where $\ifrak$ is the imaginary unit.  To ensure convergence of the deconvolution estimator, we impose the following standard condition.
\begin{assumption}\label{assumption:fourier}
    The Fourier transforms $\{\nuftn\}_{n=1}^{\infty}$ and $\nuftinfty$ are integrable.
\end{assumption}

However, since the above estimator uses the Fourier transform of the sinc kernel, it may be negative on a set of positive measure.  Therefore, the usual approach is to take the maximum with zero and renormalize the resultant density.  Though it is consistent asymptotically, in finite samples, this leads to the undesirable property that, even if $\thetahatzero$ is very close to $\thetazero$.  In particular, if the estimated prior is zero in a neighborhood of $\thetazero$, the posterior is also zero in a neighborhood of $\thetazero$ and, hence, the empirical Bayes confidence region does not cover the parameter.  Thus, for a regularizing sequence of convex weights $\{\kappan\}_{n=1}^{\infty}$ with $\kappan \to 1$, we consider 
\begin{align*}
    \pihatn(x) = \kappan \max\{\pitilden(x), 0\} + (1 - \kappan) \phi( x | \mupihat, \sigmapihatsq).
\end{align*}
Then, the posterior density is solved by numerical integration.  

In both cases, we construct the empirical Bayes confidence region as in equation \eqref{equation:definition:eb}, yielding 
\begin{align}\label{equation:npeb}
    \Icaleb \defined \Big\{ \theta \in \Theta : \pihat(\theta | \thetahatzero) > \tauhat \Big\}.
\end{align}
For this region, we have the following theorem.

\begin{theorem}\label{theorem:np}
    Suppose Assumptions \ref{assumption:berryesseen}, \ref{assumption:prior:tv}, and \ref{assumption:peb:consistent}(1) hold.
    \begin{enumerate}
        \item If, in addition, $\an = o(\sqrt{n})$, $\ansq \bnsq \to 0$, $\bnsq \to 0$, and $K \bn \to \infty$, then 
        \begin{align*}
            \Vert \pihatn - \pin \Vertone = \op(1).
        \end{align*}

        \item If, in addition, Assumption \ref{assumption:fourier} holds, then there exists a sequence $\{\bn\}_{n=1}^{\infty}$ with $\bn\sqrt{n} \to 0$ such that
        \begin{align*}
            \Vert \pihatn - \pin \Vertone = \op(1).
        \end{align*}
    \end{enumerate}

    In both settings, the region defined in equation \eqref{equation:npeb} satisfies
    \begin{align*}
        \liminf_{n \to \infty} \pr_{\thetahatzero, \thetazero}(\Icaleb \ni \thetazero) \geq 1 - \alpha.
    \end{align*}
\end{theorem}

Like Theorem \ref{theorem:parametric}, the above result shows that the nonparametric based confidence regions are also asymptotically valid.  However, unlike in the Gaussian prior setting, it is not necessary that the nonparametric empirical Bayes confidence regions have smaller Lebesgue measure than just using $\thetahatzero$ as an asymptotic pivot.  Despite this, the following hybrid approach ensures that the expected Lebesgue measure is no worse than the classical interval.

\begin{enumerate}
    \item Estimate $\pihatn(\cdot)$ as above.

    \item Simulate data according to the hierarchical model 
    \begin{align*}
        \thetahat | \theta &\sim \Ncal(\theta, \sigmahatsq / \nzero) \\ 
        \theta &\sim \pihatn
    \end{align*}
    and compute the empirical Bayes confidence regions.  

    \item Calculate the expected Lebesgue measure using the results of (2).
    \subitem If the expected Lebesgue measure is less than $2\zalpha \sqrt{\sigmahatsq / \nzero}$, apply the proposed empirical Bayes confidence region to $\thetahatzero$.
    \subitem Else, construct the classical confidence interval using $\thetahatzero$.
\end{enumerate}

We emphasize that this hybrid approach still maintains asymptotic coverage.  Since the decision to use empirical Bayes confidence regions or classical confidence intervals only depends on data independent of $\thetahatzero$ and $\thetazero$, this does not affect the coverage guaranteed by Theorem \ref{theorem:np}. 

\section{Simulations}\label{section:simulation}

In this section, we evaluate the empirical performance of our proposed procedure.  As the prototypical examples of asymptotically linear estimators, we use the least-squares estimator and the debiased lasso (cf. \cite{zhang2014confidence}, \cite{van2014asymptotically}, and \cite{javanmard2014confidence}) from Example \ref{example:linearmodel}.  In particular, for a linear model
\begin{align*}
    \yveck = \Xmatk \betaveck + \epsilonveck,
\end{align*}
where $\yveck, \epsilonveck \in \Rbb^{\nk}$, $\Xmat \in \Rbb^{\nk \times p}$, and $\betaveck \in \Rbb^{p}$, we consider $\thetak = \betavec_{k, 1}$, the first entry of $\betaveck$.  Regarding the simulation settings, we vary $K \in \{20, 50, 100\}$ and set $\nk = 100$ for $k = 1, \dots, K$.  We generate $\epsilonveck \sim \Ncal_{\nk}(\zerovec_{\nk}, \identity_{\nk})$.  For the prior distribution, we consider $\pi(\cdot) = \phi(\cdot | 1, \sigmapisq)$ and $\pi(\cdot) = 0.5 \phi(\cdot | 1, \sigmapisq) + 0.5 \phi(\cdot | -1, \sigmapisq)$ for $\sigmapisq \in \{0.1, 1\}$.

In the low-dimensional case, we let $p = 5$ and set $\nzero \in \{0, 20\}$.  Since the least-squares estimator is, in fact, linear with variance $\sigmasq = 1 / \nk$, we estimate the variance with the usual mean squared error estimator.  In the high-dimensional case, we let $p = 500$, $\sbeta = \Vert \betavec \Vertzero = 3$, and $\nzero = \{0, 100\}$.  Then, we compute the debiased lasso estimator using the {\tt hdi} package in R.

To evaluate the performance of the empirical Bayes estimators, we apply the Gaussian empirical Bayes estimator from Theorem \ref{theorem:parametric} (denoted EB-pa, where the ``pa'' stands for parametric), the nonparametric empirical Bayes estimator using kernel density (denoted EB-kd), and the nonparametric empirical Bayes estimator using deconvolution (denoted EB-dc).  As a comparison, we include the oracle Bayes estimator that has access to the true distribution $\pi(\cdot)$ (denoted OB) and the classical confidence interval using the asymptotic Gaussian approximation (denoted CL).  The three confidence regions are evaluated on the basis of coverage at a nominal level of $95\%$ and Lebesgue measure.

The results are presented in Tables \ref{tablelmsmall} -- \ref{tablehdilarge}.  In general, we notice that kernel density outperforms deconvolution when the variance of the prior distribution $\pi(\cdot)$ is small and deconvolution outperforms kernel density when the prior is more dispersed; this is consistent with the results of Section \ref{subsection:eb:np}.  Moreover, as the number of related populations increases, the performance of the empirical Bayes estimator also improves.  In general, the Lebesgue measure of the empirical Bayes confidence intervals is smaller than that of the classical confidence intervals while maintaining comparable coverage, corroborating our theoretical results.

\begin{table}[H]
\centering
\caption{Simulations for Low-Dimensional Linear Regression with $\Ncal(0, 0.1)$ Prior} 
\label{tablelmsmall}
\begin{tabular}{|l|l|rrr|rrr|}
   \hline
 &  &  \multicolumn{3}{c|}{Coverage} & \multicolumn{3}{c|}{Lebesgue Measure}  \\ 
   \hline
$\nzero$ & $K$ & 20 & 50 & 100 & 20 & 50 & 100 \\ 
   \hline
 & EB-pa & 0.918 & 0.933 & 0.945 & 1.211 & 1.207 & 1.236 \\ 
   & EB-kd & 0.952 & 0.962 & 0.970 & 1.407 & 1.388 & 1.404 \\ 
  0 & EB-dc & 0.982 & 0.985 & 0.987 & 2.178 & 2.088 & 2.040 \\ 
   & OR & 0.949 & 0.949 & 0.949 & 1.235 & 1.235 & 1.235 \\ 
   \hline
 & EB-pa & 0.920 & 0.980 & 0.940 & 0.737 & 0.766 & 0.759 \\ 
   & EB-kd & 0.850 & 0.810 & 0.820 & 0.650 & 0.618 & 0.588 \\ 
  20 & EB-dc & 0.910 & 0.980 & 0.950 & 0.764 & 0.790 & 0.779 \\ 
   & OR & 0.910 & 0.980 & 0.950 & 0.756 & 0.777 & 0.761 \\ 
   & CL & 0.950 & 0.960 & 0.950 & 0.995 & 1.038 & 0.999 \\ 
   \hline
\end{tabular}
\end{table}
\begin{table}[H]
\centering
\caption{Simulations for Low-Dimensional Linear Regression with $\Ncal(0, 1)$ Prior} 
\label{tablelmlarge}
\begin{tabular}{|l|l|rrr|rrr|}
   \hline
 &  &  \multicolumn{3}{c}{Coverage}& \multicolumn{3}{c}{Lebesgue Measure}  \\ 
   \hline
$\nzero$ & $K$ & 20 & 50 & 100 & 20 & 50 & 100 \\ 
   \hline
 & EB-pa & 0.922 & 0.934 & 0.945 & 3.807 & 3.799 & 3.899 \\ 
   & EB-kd & 0.943 & 0.953 & 0.961 & 4.212 & 4.165 & 4.227 \\ 
  0 & EB-dc & 0.900 & 0.931 & 0.944 & 3.915 & 3.939 & 4.002 \\ 
   & OR & 0.949 & 0.949 & 0.949 & 3.906 & 3.906 & 3.906 \\ 
   \hline
 & EB-pa & 0.950 & 0.950 & 0.940 & 0.948 & 0.991 & 0.960 \\ 
   & EB-kd & 0.940 & 0.950 & 0.930 & 0.936 & 0.963 & 0.927 \\ 
  20 & EB-dc & 0.870 & 0.940 & 0.920 & 0.925 & 0.969 & 0.947 \\ 
   & OR & 0.960 & 0.950 & 0.940 & 0.955 & 0.994 & 0.960 \\ 
   & CL & 0.950 & 0.960 & 0.950 & 0.995 & 1.038 & 0.999 \\ 
   \hline
\end{tabular}
\end{table}

\begin{table}[H]
\centering
\caption{Simulations for High-Dimensional Linear Regression with $\Ncal(0, 0.1)$ Prior} 
\label{tablehdismall}
\begin{tabular}{|l|l|rrr|rrr|}
   \hline
 &  &  \multicolumn{3}{c|}{Coverage} & \multicolumn{3}{c|}{Lebesgue Measure}  \\ 
   \hline
$\nzero$ & $K$ & 20 & 50 & 100 & 20 & 50 & 100 \\ 
   \hline
 & EB-pa & 0.914 & 0.940 & 0.944 & 1.209 & 1.234 & 1.231 \\ 
   & EB-kd & 0.952 & 0.970 & 0.971 & 1.423 & 1.426 & 1.409 \\ 
  0 & EB-dc & 0.988 & 0.991 & 0.992 & 2.499 & 2.461 & 2.455 \\ 
   & OR & 0.949 & 0.949 & 0.949 & 1.235 & 1.235 & 1.235 \\ 
   \hline
 & EB-pa & 0.920 & 0.940 & 0.973 & 0.426 & 0.422 & 0.425 \\ 
   & EB-kd & 0.913 & 0.947 & 0.927 & 0.417 & 0.409 & 0.402 \\ 
  100 & EB-dc & 0.900 & 0.947 & 0.973 & 0.434 & 0.429 & 0.433 \\ 
   & OR & 0.927 & 0.947 & 0.973 & 0.431 & 0.423 & 0.426 \\ 
   & CL & 0.913 & 0.960 & 0.980 & 0.462 & 0.452 & 0.456 \\ 
   \hline
\end{tabular}
\end{table}
\begin{table}[H]
\centering
\caption{Simulations for High-Dimensional Linear Regression with $\Ncal(0, 1)$ Prior} 
\label{tablehdilarge}
\begin{tabular}{|l|l|rrr|rrr|}
   \hline
 &  &  \multicolumn{3}{c|}{Coverage} & \multicolumn{3}{c|}{Lebesgue Measure}  \\ 
   \hline
$\nzero$ & $K$ & 20 & 50 & 100 & 20 & 50 & 100 \\ 
   \hline
 & EB-pa & 0.919 & 0.940 & 0.943 & 3.797 & 3.868 & 3.854 \\ 
   & EB-kd & 0.940 & 0.959 & 0.958 & 4.210 & 4.213 & 4.150 \\ 
  0 & EB-dc & 0.909 & 0.940 & 0.944 & 4.005 & 4.042 & 3.962 \\ 
   & OR & 0.949 & 0.949 & 0.949 & 3.906 & 3.906 & 3.906 \\ 
   \hline
 & EB-pa & 0.940 & 0.940 & 0.953 & 0.454 & 0.446 & 0.450 \\ 
   & EB-kd & 0.940 & 0.933 & 0.953 & 0.454 & 0.447 & 0.449 \\ 
  100 & EB-dc & 0.907 & 0.913 & 0.953 & 0.448 & 0.446 & 0.451 \\ 
   & OR & 0.940 & 0.940 & 0.953 & 0.455 & 0.447 & 0.450 \\ 
   & CL & 0.940 & 0.947 & 0.953 & 0.459 & 0.451 & 0.454 \\ 
   \hline
\end{tabular}
\end{table}

\medskip

\section{Application to the Trends in International Mathematics and Sciences Study (TIMSS)}\label{section:realdata}

The Trends in International Mathematics and Sciences Study is an international study conducted every four years to measure fourth and eighth grade achievement in mathematics and sciences.  Polities sample representative schools, and the students within the schools take standardized examinations in mathematics and science.  The data is freely available at \url{https://timssandpirls.bc.edu/}; in addition to an overall measure of achievement for each school, we also have access to various school background covariates.  These include, for example, the total enrollment of the school, the amount of students from economically disadvantaged backgrounds, amongst others.  For our analysis, we focus only on fourth grade achievement in mathematics in 2015.  For a more detailed description of the methodology of TIMSS and a general overview, we refer the interested reader to \cite{martin2016} and \cite{mullis20years} respectively.

We are interested in seeing which school background covariates significantly influence student achievement (i) in the United States by leveraging the observations in other polities and (ii) in a new, unobserved polity.  In the context of our first problem, we view polities as the various populations $\Pzero, \Pone, \dots, \PK$, with $\Pzero$ denoting the United States.  Then, within each polity, the schools comprise our observational units.  Since some covariates have very low variability within a polity, such as the amount of digital magazines with different titles, we omit those covariates; this leaves us with $p = 90$ distinct covariates.  Moreover, we only consider polities with $\nk \geq 50$ schools for a total of $K = 34$.  Regarding the model, we assume a high-dimensional linear model and apply the debiased lasso.  For the unobserved polity, the formulation is nearly identical, except we have $K = 35$ distinct polities from which we estimate the distribution.

We only apply deconvolution since we believe the variation amongst polities to be of larger order than the noise in estimating the parameters due to the relatively small sample size.  Moreover, due to geographic and economic factors, we do not believe that the polities are necessarily Gaussian distributed.  We consider each variable marginally, without correcting for multiple testing.  All results are presented at the $\alpha = 0.05$ level.  As a point of comparison, when applying the debiased lasso on just the United States, there are three significant covariates:  (i) the amount of students receiving free lunch with a p-value of $0.0223$, (ii) the amount of print books with different titles with a p-value of $0.0418$, and (iii) the principal's report on how the school's capacity to provide instruction was affected by resource shortage with a p-value of $0.0177$.  However, when using deconvlution, we see that the amount of students receiving free lunch is no longer significant (p-value of $0.0564$), but both of the other two covariates continue to have a significant, positive effect (p-values of $0.0272$ and $0.0355$ respectively).  Although our empirical Bayes confidence regions are strictly shorter than the classical confidence intervals, the regions are not necessarily contained in each other.  For the amount of students receiving free lunch, both are in fact intervals and a $95\%$ empirical Bayes confidence region is $(-0.2646, 18.9489)$ while a classical confidence interval is $(0.3749, 19.8356)$.  

On the other hand, for a new polity, we rank the covariates by the maximum of the estimated probabilities of having a positive effect and of having a negative effect.  Then, the principal's report on how problematic school fights are has the highest probability of having a negative impact on student performance, with an estimated probability greater than $75\%$.  This suggests that reducing school fights is likely to improve student performance in a polity that did not participate in TIMSS.

\section{Proofs}
\label{section:appendix}

Here, we provide the proof of all of the results along with supplemental lemmata.

\begin{proof}[Proof of Proposition \ref{proposition:oracle:length}]
    Indeed, consider the optimization problem 
    \begin{align*}
        \minimize  &\phantleq \Ebb_{\Umatzero, \thetazero} \lambda(\Ical) \\ 
        \subjectto &\phantleq \pr_{\Umatzero, \thetazero} (\Ical \ni \thetazero) \geq 1 - \alpha.
    \end{align*}
    Expanding the first term, we have
    \begin{align*}
        \Ebb_{\Umatzero, \thetazero} \lambda(\Ical) 
        &= \int_{\Rbb^{\uzero}} f(\Umatzero) \lambda(d\Umatzero) \int_{\Theta} \lambda(\Ical) \pi(\thetazero | \Umatzero) \lambda(d\thetazero) \\ 
        &= \int_{\Rbb^{\uzero}} f(\Umatzero) \lambda(d\Umatzero) \int_{\Theta} \pi(\thetazero | \Umatzero) \lambda(d\thetazero) \int_{\Theta} \indic{\Ical \ni \varthetazero} \lambda(d\varthetazero). 
    \end{align*}
    Similarly, expanding the second term yields
    \begin{align*}
        \pr_{\Umatzero, \thetazero} (\Ical \ni \thetazero) 
        &= \int_{\Rbb^{\uzero}} f(\Umatzero) \lambda(d\Umatzero) \int_{\Theta} \indic{\Ical \ni \thetazero} \pi(\thetazero | \Umatzero) \lambda(d\thetazero) \\ 
        &= \int_{\Rbb^{\uzero}} f(\Umatzero) \lambda(d\Umatzero) \int_{\Theta} \indic{\Ical \ni \varthetazero} \pi(\varthetazero | \Umatzero) \lambda(d\varthetazero) \\ 
        &= \int_{\Rbb^{\uzero}} f(\Umatzero) \lambda(d\Umatzero) \int_{\Theta} \pi(\thetazero | \Umatzero) \lambda(d\thetazero) \int_{\Theta} \indic{\Ical \ni \varthetazero} \pi(\varthetazero | \Umatzero) \lambda(d\varthetazero) 
    \end{align*}
    where we change the variable of integration $\thetazero \mapsto \varthetazero$ in the second line and use the fact that $\pi(\thetazero | \Umatzero)$ integrates to one in the last line.  Now, letting $\gamma > 0$ be a Lagrange multiplier and combining the above calculations, we have 
    \begin{align*}
        &\Ebb_{\Umatzero, \thetazero} \lambda(\Ical) - \gamma \pr_{\Umatzero, \thetazero} (\Ical \ni \thetazero) \\
        &\phantleq = \int_{\Rbb^{\uzero}} f(\Umatzero) \lambda(d\Umatzero) \int_{\Theta} \pi(\thetazero | \Umatzero) \lambda(d\thetazero) \int_{\Theta} \indic{\Ical \ni \varthetazero} \Big( 1 - \gamma \pi(\varthetazero | \Umatzero) \Big) \lambda(d\varthetazero).
    \end{align*}
    From here, it is immediate that the above display is minimized when 
    \begin{align*}
        \Ical 
        = \Big\{ \vartheta \in \Theta : 1 - \gamma \pi(\vartheta | \Umatzero) < 0 \Big\}
        = \Big\{ \vartheta \in \Theta : \pi(\vartheta | \Umatzero) > \frac{1}{\gamma} \Big\}
    \end{align*}
    Finally, noting that $\gamma = 1/\tau$ is the smallest value satisfying the constraint by construction and $\lambda(\Ical)$ is monotonically increasing in $\gamma$ finishes the proof.
\end{proof}

\begin{proof}[Proof of Proposition \ref{proposition:totalvariation}]
    Indeed, we have 
    \begin{align*}
        &\int_{\Theta} \lambda(d\thetazero) \int_{\Rbb^{\uzero}} \Big| \fn(\Umatzero | \thetazero) \pin(\thetazero) - \fhatn(\Umatzero | \thetazero) \pihatn(\thetazero) \Big| \lambda(d\Umatzero) \\
        &\phantleq \leq \int_{\Theta} \lambda(d\thetazero) \int_{\Rbb^{\uzero}} \Big| \fn(\Umatzero | \thetazero) \pin(\thetazero) - \fn(\Umatzero | \thetazero) \piinfty(\thetazero) \Big| \lambda(d\Umatzero) \\
        &\phantleq \phantleq + \int_{\Theta} \lambda(d\thetazero) \int_{\Rbb^{\uzero}} \Big| \fn(\Umatzero | \thetazero) \piinfty(\thetazero) - \fhatn(\Umatzero | \thetazero) \piinfty(\thetazero) \Big| \lambda(d\Umatzero) \\ 
        &\phantleq \phantleq + \int_{\Theta} \lambda(d\thetazero) \int_{\Rbb^{\uzero}} \Big| \fhatn(\Umatzero | \thetazero) \piinfty(\thetazero) - \fhatn(\Umatzero | \thetazero) \pin(\thetazero) \Big| \lambda(d\Umatzero) \\ 
        &\phantleq \phantleq + \int_{\Theta} \lambda(d\thetazero) \int_{\Rbb^{\uzero}} \Big| \fhatn(\Umatzero | \thetazero) \pin(\thetazero) - \fhatn(\Umatzero | \thetazero) \pihatn(\thetazero) \Big| \lambda(d\Umatzero).
    \end{align*}
    We consider each of the four terms on the right hand side separately.  First, note that 
    \begin{align*}
        &\int_{\Theta} \lambda(d\thetazero) \int_{\Rbb^{\uzero}} \Big| \fn(\Umatzero | \thetazero) \pin(\thetazero) - \fn(\Umatzero | \thetazero) \piinfty(\thetazero) \Big| \lambda(d\Umatzero) \\ 
        &\phantleq = \int_{\Theta} |\pin(\thetazero) - \piinfty(\thetazero)| \lambda(d\thetazero) \int_{\Rbb^{\uzero}} \fn(\Umatzero | \thetazero) \lambda(d\Umatzero) \\ 
        &\phantleq = \int_{\Theta} |\pin(\thetazero) - \piinfty(\thetazero)| \lambda(d\thetazero) \\ 
        &\phantleq = \Vert \pin(\cdot) - \piinfty(\cdot) \Vertone.
    \end{align*}
    The above converges to zero as $\pin$ converges to $\piinfty$ in total variation distance.  The proof for the third and fourth terms are analogous and, thus, omitted.
    
    For the second term, let $(\Upsilonj)_{j=1}^{\infty} \subseteq \Theta$ be a collection of compact sets such that $\cup_{j=1}^{\infty} \Upsilonj = \Theta$ and $\lambda(\Upsilonj \cap \Upsilon_{j'}) = 0$ for $j \neq j'$.  Note that 
    \begin{align*}
        \sum_{j=1}^{\infty} \int_{\Upsilonj} \pi(\thetazero) \lambda(d\thetazero) = \int_{\Theta} \pi(\thetazero) \lambda(d\thetazero) = 1.
    \end{align*}
    Now, fix a value of $\epsilon > 0$ and choose $J$ sufficiently large such that 
    \begin{align*}
        \sum_{j=J+1}^{\infty} \int_{\Upsilonj} \pi(\thetazero) \lambda(d\thetazero) < \epsilon / 4.
    \end{align*}
    Then, 
    \begin{align*}
        &\int_{\Theta} \lambda(d\thetazero) \int_{\Rbb^{\uzero}} \Big| \fn(\Umatzero | \thetazero) \piinfty(\thetazero) - \fhatn(\Umatzero | \thetazero) \piinfty(\thetazero) \Big| \lambda(d\Umatzero) \\ 
        &\phantleq \leq \int_{\Theta} \piinfty(\thetazero) \lambda(d\thetazero) \int_{\Rbb^{\uzero}} |\fn(\Umatzero | \thetazero) - \fhatn(\Umatzero | \thetazero)| \lambda(d\Umatzero) \\ 
        &\phantleq = \sum_{j=1}^{J} \int_{\Upsilonj} \piinfty(\thetazero) \lambda(d\thetazero) \int_{\Rbb^{\uzero}} |\fn(\Umatzero | \thetazero) - \fhatn(\Umatzero | \thetazero)| \lambda(d\Umatzero) \\ 
        &\phantleq \phantleq+ \sum_{j=J+1}^{\infty} \int_{\Upsilonj} \piinfty(\thetazero) \lambda(d\thetazero) \int_{\Rbb^{\uzero}} |\fn(\Umatzero | \thetazero) - \fhatn(\Umatzero | \thetazero)| \lambda(d\Umatzero)
    \end{align*}
    Letting $N$ be sufficiently large such that for all $n \geq N$
    \begin{align*}
        \max_{j=1, \dots, J} \sup_{\thetazero \in \Upsilonj} \int_{\Rbb^{\uzero}} |\fn(\Umatzero | \thetazero) - \fhatn(\Umatzero | \thetazero)| \lambda(d\Umatzero) < \epsilon / 2,
    \end{align*}
    we have 
    \begin{align*}
        \sum_{j=1}^{J} \int_{\Upsilonj} \piinfty(\thetazero) \lambda(d\thetazero) \int_{\Rbb^{\uzero}} |\fn(\Umatzero | \thetazero) - \fhatn(\Umatzero | \thetazero)| \lambda(d\Umatzero) < (1 - \epsilon)\epsilon / 2 < \epsilon / 2.
    \end{align*}
    Likewise, since 
    \begin{align*}
        \int_{\Rbb^{\uzero}} |\fn(\Umatzero | \thetazero) - \fhatn(\Umatzero | \thetazero)| \lambda(d\Umatzero) \leq 2,
    \end{align*}
    it follows that 
    \begin{align*}
        \sum_{j=J+1}^{\infty} \int_{\Upsilonj} \piinfty(\thetazero) \lambda(d\thetazero) \int_{\Rbb^{\uzero}} |\fn(\Umatzero | \thetazero) - \fhatn(\Umatzero | \thetazero)| \lambda(d\Umatzero) < \epsilon / 2,
    \end{align*}
    implying further that 
    \begin{align*}
        \int_{\Theta} \lambda(d\thetazero) \int_{\Rbb^{\uzero}} \Big| \fn(\Umatzero | \thetazero) \piinfty(\thetazero) - \fhatn(\Umatzero | \thetazero) \piinfty(\thetazero) \Big| \lambda(d\Umatzero) < \epsilon.
    \end{align*}
    Since $\epsilon > 0$ is arbitrary, we conclude that 
    \begin{align*}
        \limsup_{n \to \infty} \int_{\Theta} \lambda(d\thetazero) \int_{\Rbb^{\uzero}} \Big| \fn(\Umatzero | \thetazero) \piinfty(\thetazero) - \fhatn(\Umatzero | \thetazero) \piinfty(\thetazero) \Big| \lambda(d\Umatzero) = 0.
    \end{align*}
    Combining the above calculations finishes the proof.
\end{proof}

\begin{proof}[Proof of Corollary \ref{corollary:coverage}]
    The result is an immediate consequence of Scheff\'e's Theorem and Proposition \ref{proposition:totalvariation}.
\end{proof}

\begin{proof}[Proof of Proposition \ref{proposition:fn}]
    Indeed, let $\hn(\cdot | \thetan)$ denote the density of $\sqrt{n} (\thetahatn - \thetan ) | \thetan$.  Then, by a change of variables, it follows that 
    \begin{align*}
        \Vert \fn(\cdot | \thetan) - \phi(\cdot | \sqrt{n}\thetan, \sigmasq) \Vertone
        = \Vert \hn(\cdot | \thetan) - \phi(\cdot | 0, \sigmasq) \Vertone.
    \end{align*}
    Furthermore, by construction, letting $\htilden(\cdot | \theta)$ denote the density of $\sqrt{n} (\thetatilden - \thetan) | \thetan$, we have 
    \begin{align*}
        \hn = \htilden \ast \phi_{0, \varsigmasqn}.
    \end{align*}
    Thus, the triangle inequality yields
    \begin{align*}
        \Vert \hn(\cdot | \thetan) - \phi(\cdot | 0, \sigmasq) \Vertone
        \leq \Vert (\htilden - \phi_{0, \sigmasq}) \ast \phi_{0, \varsigmasqn} \Vertone + \Vert \phi_{0, \sigmasq} \ast \phi_{0, \varsigmasqn} - \phi_{0, \sigmasq} \Vertone.
    \end{align*}
    Note that $\phi_{0, \varsigmasqn}$ is an approximate convolutional identity as $\varsigmasqn \to 0$, implying that 
    \begin{align*}
        \Vert \phi_{0, \sigmasq} \ast \phi_{0, \varsigmasqn} - \phi_{0, \sigmasq} \Vertone = o(1)
    \end{align*}
    (for example, see Theorem 4.2.4 of \cite{bogachev2007measure}).  As the above does not depend on $\theta \in \Theta$, the convergence is uniform over compacts $\Upsilon \subseteq \Theta$; that is, 
    \begin{align*}
        \sup_{\theta \in \Upsilon} \Vert \phi_{0, \sigmasq} \ast \phi_{0, \varsigmasqn} - \phi_{0, \sigmasq} \Vertone = o(1).
    \end{align*}
    Now, for the other term, we have 
    \begin{align*}
        \Vert (\htilden - \phi_{0, \sigmasq}) \ast \phi_{0, \varsigmasqn} \Vertinfty 
        = \Vert (\Htilden - \Phi_{0, \sigmasq}) \ast \phi_{0, \varsigmasqn}' \Vertinfty
        \leq \Vert \Htilden - \Phi_{0, \sigmasq} \Vertinfty \Vert \phi_{0, \varsigmasqn}' \Vertone,
    \end{align*}
    where $\Htilden(\cdot | \theta)$ denotes the distribution function corresponding to $\htilden(\cdot | \theta)$.  The inequality in the above display is a consequence of Young's convolution inequality.  Let $\{\ratedelta\}_{n=1}^{\infty}$ be a positive sequences of numbers decreasing to zero such that 
    \begin{align*}
        \pr\Big( |\sqrt{n} \deltan| \geq \ratedelta \Big) \leq \ratedelta.
    \end{align*}
    Such a sequence $\ratedelta$ exists as $\deltan = \op(n^{-1/2})$.  Then, for any $t \in \Rbb$, we have
    \begin{align*}
        \Htilden(t) 
        = \pr\Big( \sqrt{n} (\thetatilden - \thetan) + \deltan \leq t \Big)
        = \pr\Big( \frac{1}{\sqrt{n}} \sum_{i=1}^{n} \psi(X_{i} ; \thetan, \etan) + \sqrt{n} \deltan \leq t \Big).
    \end{align*}
    Note that
    \begin{align*}
        \Big\{ \frac{1}{\sqrt{n}} \sum_{i=1}^{n} \psi(X_{i} ; \thetan, \etan) \leq t - \varepsilonn \Big\}
        \cap \Big\{ |\sqrt{n} \deltan| \leq \varepsilonn \Big\} 
        \subseteq
        \Big\{ \frac{1}{\sqrt{n}} \sum_{i=1}^{n} \psi(X_{i} ; \thetan, \etan) + \sqrt{n} \deltan \leq t \Big\} 
    \end{align*}
    and 
    \begin{align*}
        \Big\{ \frac{1}{\sqrt{n}} \sum_{i=1}^{n} \psi(X_{i} ; \thetan, \etan) + \sqrt{n} \deltan \leq t \Big\} 
        \subseteq 
        \Big\{ \frac{1}{\sqrt{n}} \sum_{i=1}^{n} \psi(X_{i} ; \thetan, \etan) \leq t + \varepsilonn \Big\}
        \cup 
        \Big\{ |\sqrt{n} \deltan| > \varepsilonn \Big\},
    \end{align*}
    which implies
    \begin{align*}
        &\pr\Big( \frac{1}{\sqrt{n}} \sum_{i=1}^{n} \psi(X_{i} ; \thetan, \etan) \leq t - \varepsilonn \Big) 
        + \pr\Big( |\sqrt{n} \deltan| \leq \varepsilonn \Big) - 1 \\ 
        &\phantleq \leq 
        \pr\Big( \frac{1}{\sqrt{n}} \sum_{i=1}^{n} \psi(X_{i} ; \thetan, \etan) + \sqrt{n} \deltan \leq t \Big)
        \\ &\phantleq \phantleq \leq 
        \pr\Big( \frac{1}{\sqrt{n}} \sum_{i=1}^{n} \psi(X_{i} ; \thetan, \etan) \leq t + \varepsilonn \Big) 
        + \pr\Big( |\sqrt{n} \deltan| > \varepsilonn \Big).
    \end{align*}
    Therefore, 
    \begin{align*}
        &\Vert \Htilden - \Phi_{0, \sigmasq} \Vertinfty  \\ 
        &\phantleq= \sup_{t \in \Rbb} \Big| \pr\Big( \frac{1}{\sqrt{n}} \sum_{i=1}^{n} \psi(X_{i} ; \thetan, \etan) + \sqrt{n} \deltan \leq t \Big) - \Phi_{0, \sigmasq}(t)\Big| \\ 
        &\phantleq\leq \max \Big\{ \sup_{t\in \Rbb} \Big| \pr\Big( \frac{1}{\sqrt{n}} \sum_{i=1}^{n} \psi(X_{i} ; \thetan, \etan) \leq t + \varepsilonn \Big) 
        + \pr\Big( |\sqrt{n} \deltan| > \varepsilonn \Big) - \Phi_{0, \sigmasq}(t) \Big|, \\
        &\phantleq \phantom{\leq \max \Big\{ \leq } 
        \sup_{t\in \Rbb} \Big| \pr\Big( \frac{1}{\sqrt{n}} \sum_{i=1}^{n} \psi(X_{i} ; \thetan, \etan) \leq t - \varepsilonn \Big) 
        + \pr\Big( |\sqrt{n} \deltan| \leq \varepsilonn \Big) - 1 - \Phi_{0, \sigmasq}(t) \Big| \Big\}.
    \end{align*}
    Now, by the mean value theorem, we see that, for any $s \in \Rbb$, 
    \begin{align*}
        |\Phi_{0, \sigmasq}(t) - \Phi_{0, \sigmasq}(s)| \leq \frac{|t - s|}{\sqrt{2\pi\sigmasq}}.
    \end{align*}
    Thus, it follows from the Berry-Esseen theorem that
    \begin{align*}
        &\sup_{t \in \Rbb} \Big| \pr\Big( \frac{1}{\sqrt{n}} \sum_{i=1}^{n} \psi(X_{i} ; \thetan, \etan) \leq t + \varepsilonn \Big) + \pr\Big( |\sqrt{n} \deltan | > \varepsilonn \Big) - \Phi_{0, \sigmasq}(t) \Big| \\ 
        &\phantleq\leq \sup_{t \in \Rbb} \Big| \pr\Big( \frac{1}{\sqrt{n}} \sum_{i=1}^{n} \psi(X_{i} ; \thetan, \etan) \leq t + \varepsilonn \Big)  - \Phi_{0, \sigmasq}(t + \varepsilonn) \Big|
        \\ &\phantleq \phantleq + \sup_{t \in \Rbb} \Big|\Phi_{0, \sigmasq}(t + \varepsilonn) - \Phi_{0, \sigmasq}(t) \Big| +\sup_{t \in \Rbb} \pr\Big( |\sqrt{n} \deltan | > \varepsilonn \Big)\\ 
        &\phantleq = \Ocal\Big( \frac{1}{\sqrt{n}} + \varepsilonn \Big).
    \end{align*}
    Similarly, we have 
    \begin{align*}
        &\sup_{t\in \Rbb} \Big| \pr\Big( \frac{1}{\sqrt{n}} \sum_{i=1}^{n} \psi(X_{i} ; \thetan, \etan) \leq t - \varepsilonn \Big) 
        + \pr\Big( |\sqrt{n} \deltan| \leq \varepsilonn \Big) - 1 - \Phi_{0, \sigmasq}(t) \Big| \\ 
        &\phantleq\leq \sup_{t\in \Rbb} \Big| \pr\Big( \frac{1}{\sqrt{n}} \sum_{i=1}^{n} \psi(X_{i} ; \thetan, \etan) \leq t - \varepsilonn \Big)  - \Phi_{0, \sigmasq}(t - \varepsilonn) \Big|
        \\ &\phantleq \phantleq + \sup_{t\in \Rbb} \Big|\Phi_{0, \sigmasq}(t - \varepsilonn) - \Phi_{0, \sigmasq}(t) \Big| + \sup_{t\in \Rbb} \pr\Big( |\sqrt{n} \deltan | > \varepsilonn \Big)\\ 
        &\phantleq = \Ocal\Big( \frac{1}{\sqrt{n}} + \varepsilonn \Big).
    \end{align*}
    From Assumption \ref{assumption:berryesseen}, since the third moment is uniformly bounded on compacts, the convergence is uniform.  Hence, we conclude that 
    \begin{align}\label{equation:uclt}
        \sup_{\theta \in \Upsilon} \Vert \Htilden - \Phi_{0, \sigmasq} \Vertinfty = \Ocal\Big(\frac{1}{\sqrt{n}} + \varepsilonn \Big)
    \end{align}
    for all compacts $\Upsilon \subseteq \Theta$.  
    Moreover, a direct calculation shows that 
    \begin{align*}
        \Vert \phi_{0, \varsigmasqn}' \Vertone 
        = \int_{-\infty}^{\infty} \frac{|\vartheta|}{\varsigmasqn \sqrt{2\pi \varsigmasqn}} \exp\Big( -\frac{\vartheta^{2}}{2\varsigmasqn} \Big) \lambda(d\vartheta)
        = \frac{1}{\varsigma_{n}}\sqrt{\frac{2}{\pi}}.
    \end{align*}
    Combining the above calculations yields
    \begin{align*}
        \sup_{\theta \in \Upsilon} \Vert (\htilden - \phi_{0, \sigmasq}) \ast \phi_{0, \varsigmasqn} \Vertinfty = o(1)
    \end{align*}
    for all compacts $\Upsilon \subseteq \Theta$.  Writing $\htilde_{n, \thetan}(\cdot) = \htilden(\cdot | \thetan)$, the above implies that $\{\htilde_{n, \thetan} \ast \phi_{0, \varsigmasqn}\}_{n \in \Nbb, \theta \in \Upsilon}$ is a tight family of probability measures for all compacts $\Upsilon \subseteq \Theta$.  It is left to show that 
    \begin{align*}
        \sup_{\theta \in \Upsilon} \Vert (\htilden - \phi_{0, \sigmasq}) \ast \phi_{0, \varsigmasqn} \Vertone = o(1).
    \end{align*}
    Thus, for an arbitrary value of $\epsilon > 0$, let $\Omega_{\epsilon} \subseteq \Rbb$ be a compact set such that 
    \begin{align*}
        \inf_{n} \Big\{ \inf_{\theta \in \Upsilon} \int_{\Omega_{\epsilon}^\comp} (\htilde_{n, \thetan} \ast \phi_{0, \varsigmasqn})(t) \lambda(dt) + \int_{\Omega_{\epsilon}^\comp} \phi_{0, \sigmasq + \varsigmasqn}(t) \lambda(dt) \Big\} < \epsilon / 2.
    \end{align*}
    Set $N$ sufficiently large such that for all $n \geq N$
    \begin{align*}
        \sup_{\theta \in \Upsilon} \Vert (\htilden - \phi_{0, \sigmasq}) \ast \phi_{0, \varsigmasqn} \Vertinfty < \epsilon / (2 \lambda(\Omega_{\epsilon})). 
    \end{align*}
    Then, we have 
    \begin{align*}
        \sup_{\theta \in \Upsilon} \Vert (\htilden - \phi_{0, \sigmasq}) \ast \phi_{0, \varsigmasqn} \Vertone
        &= \sup_{\theta \in \Upsilon} \int_{\Omega_{\epsilon}} |[(\htilden - \phi_{0, \sigmasq}) \ast \phi_{0, \varsigmasqn}](t)| \lambda(dt) 
        + \sup_{\theta \in \Upsilon} \int_{\Omega_{\epsilon}^\comp} |[(\htilden - \phi_{0, \sigmasq}) \ast \phi_{0, \varsigmasqn}](t)| \lambda(dt) \\ 
        &\leq \epsilon.
    \end{align*}
    Combining the above calculations finishes the proof.
\end{proof}

\begin{lemma}\label{lemma:gaussian:tv}
    For $\muone, \mutwo \in \Rbb$ and $\sigmaonesq, \sigmatwosq > 0$, the following bound holds:
    \begin{align*}
        \Vert \phi_{\muone, \sigmaonesq} - \phi_{\mutwo, \sigmatwosq} \Vertone^{2} \leq \frac{1}{2} \log\frac{\sigmatwosq}{\sigmaonesq} + \frac{\sigmaonesq + (\muone - \mutwo)^{2}}{2\sigmatwosq} - \frac{1}{2}.
    \end{align*}
\end{lemma}
\begin{proof}[Proof of Lemma \ref{lemma:gaussian:tv}]
    The result follows from Pinsker's inequality.
\end{proof}

\begin{proof}[Proof of Corollary \ref{corollary:fn}]
    From Proposition \ref{proposition:fn}, it suffices to show that 
    \begin{align*}
        \Vert \phi_{\sqrt{n} \thetan, \sigmasq} - \phi_{\sqrt{n} \thetan, \sigmahatsq} \Vertone = \op(1).
    \end{align*}
    The result now follows from Lemma \ref{lemma:gaussian:tv}.
\end{proof}

\begin{lemma}\label{lemma:ltwo:lone}
    Let $1 \leq q < p \leq \infty$.  If $\{\gn\}_{n=1}^{\infty}$ and $g(\cdot)$ are probability densities such that $\Vert \gn - g \Vertp \to 0$, then $\Vert \gn - g \Vertq \to 0$.  
\end{lemma}

\begin{proof}[Proof of Lemma \ref{lemma:ltwo:lone}]
    We start by showing the result holds if $1 = q < p \leq \infty$.  It suffices to show that each subsequence $\{\nk\}_{k=1}^{\infty}$ contains a further subsequence $\{\nkl\}_{l=1}^{\infty}$ such that $\Vert g_{\nkl} - g \Vertone \to 0$.  To this end, let $\{\nk\}_{k=1}^{\infty}$ be an arbitrary subsequence.  Then, since $\Vert g_{\nk} - g \Vertp \to 0$, there exists a further subsequence $\{\nkl\}_{l=1}^{\infty}$ such that $g_{\nkl} \to g$ almost everywhere.  Thus, by Scheff\'e's lemma, it follows that $\Vert g_{\nkl} - g \Vertone \to 0$.

    Now, suppose instead that $q > 1$.  Expanding out the integral, we have
    \begin{align*}
        \Vert \gn - g \Vertq^{q} 
        &= \int_{-\infty}^{\infty} |\gn(x) - g(x)|^{q} \lambda(dx) \\
        &= \int_{|\gn(x) - g(x)| > 1} |\gn(x) - g(x)|^{q} \lambda(dx) + \int_{|\gn(x) - g(x)| < 1} |\gn(x) - g(x)|^{q} \lambda(dx).
    \end{align*}
    If $p < \infty$, then 
    \begin{align*}
        \Vert \gn - g \Vertq^{q} 
        &\leq \int_{|\gn(x) - g(x)| > 1} |\gn(x) - g(x)|^{p} \lambda(dx) + \int_{|\gn(x) - g(x)| < 1} |\gn(x) - g(x)| \lambda(dx) \\ 
        &\leq \Vert \gn - g \Vertp^{p} + \Vert \gn - g \Vertone \\ 
        &= o(1)
    \end{align*}
    by the previous case.  If $p = \infty$, there exists an $N$ sufficiently large such that $\Vert \gn - g \Vertinfty < 1$ for all $n \geq N$.  Thus, 
    \begin{align*}
        \int_{|\gn(x) - g(x)| > 1} |\gn(x) - g(x)|^{q} \lambda(dx) = 0
    \end{align*}
    for all $n \geq N$.  This finishes the proof.
\end{proof}

\begin{proof}[Proof of Proposition \ref{proposition:qconvergence}]
    Without a loss of generality, by recentering our distribution, we assume that $\Ebb \thetan = 0$ for all $n$.  Expanding the definition, we see that 
    \begin{align*}
        \an \thetahatn 
        = \an \thetatilden + \frac{\an \xin}{n}
        = \an \thetan + \frac{\an}{n} \sum_{i=1}^{n} \psi(X_{i} ; \thetan, \etan) + \an \deltan + \frac{\an \xin}{n}.
    \end{align*}
    We start by considering the first assertion.  If $\an = o(\sqrt{n})$, then 
    \begin{align*}
        \an \thetahatn = \an \thetan + \op(1) + \frac{\an \xin}{n}.
    \end{align*}
    By the triangle inequality, we have that 
    \begin{align*}
        \Vert \qn - \nuinfty \Verttwo 
        \leq \Vert \qn(t) - \nuinfty \ast \phi_{0, \ansq \varsigmasqn / n} \Verttwo 
        + \Vert \nuinfty \ast \phi_{0, \ansq \varsigmasqn / n} - \nuinfty \Verttwo.
    \end{align*}
    If $\ansq \varsigmasqn / n \to 0$, it follows by Theorem 4.2.4 of \cite{bogachev2007measure} that 
    \begin{align*}
        \Vert \nuinfty \ast \phi_{0, \ansq \varsigmasqn / n} - \nuinfty \Verttwo = o(1).
    \end{align*}
    Let $\Qtilden(\cdot)$ and $\qtilden(\cdot)$ denote the marginal distribution function and marginal density of $\an \thetatilden$ respectively.  Now, Young's convolution inequality yields
    \begin{align*}
        \Vert \qn - \nuinfty \ast \phi_{0, \ansq \varsigmasqn / n} \Vertinfty 
        = \Vert (\qtilden - \nuinfty) \ast \phi_{0, \ansq \varsigmasqn / n} \Vertinfty
        = \Vert (\Qtilden - \Nuinfty) \ast \phi_{0, \ansq \varsigmasqn / n}' \Vertinfty 
        \leq \Vert \Qtilden - \Nuinfty \Vertinfty \Vert \phi_{0, \ansq \varsigmasqn / n}' \Vertone.
    \end{align*}
    Since $\Nuinfty$ is a continuous distribution, the convergence of $\Qtilden$ to $\Nuinfty$ is uniform (for example, see Exercise 8.1.13 of \cite{chow1997probability}), implying that $\Vert \Qtilden - \Nuinfty \Vertinfty = o(1)$.  Moreover, since 
    \begin{align*}
        \Vert \phi_{0, \ansq \varsigmasqn / n}' \Vertone = \sqrt{\frac{2n}{\pi \ansq \varsigmasqn}}, 
    \end{align*} there exists a sequence $\{\varsigmasqn\}_{n=1}^{\infty}$ such that (i) $\ansq \varsigmasqn / n \to 0$ and (ii)
    \begin{align*}
        \Vert \Qtilden - \Nuinfty \Vertinfty \Vert \phi_{0, \ansq \varsigmasqn / n}' \Vertone = o(1).
    \end{align*}
    With this choice of $\varsigmasqn$, the first claim follows by Lemma \ref{lemma:ltwo:lone}.
    
    Next, suppose instead that $\an = \sqrt{n}$.  Again, we have the decomposition
    \begin{align*}
        \Vert \qn - \nuinfty \ast \phi_{0, \sigmasq} \Verttwo 
        &\leq \Vert \qn - \nun \ast \phi_{0, \sigmasq + \varsigmasqn} \Verttwo 
        + \Vert \nun \ast \phi_{0, \sigmasq + \varsigmasqn} - \nuinfty \ast \phi_{0, \sigmasq + \varsigmasqn} \Verttwo \\
        &\phantleq+ \Vert \nuinfty \ast \phi_{0, \sigmasq + \varsigmasqn} - \nuinfty \ast \phi_{0, \sigmasq} \Verttwo.
    \end{align*}
    Theorem 4.5.4 of \cite{bogachev2007measure} immediately implies $\Vert \nuinfty \ast \phi_{0, \sigmasq + \varsigmasqn} - \nuinfty \ast \phi_{0, \sigmasq} \Verttwo = o(1)$.  Moreover, by Assumption \ref{assumption:prior:tv}, it follows from Young's convolution inequality that 
    \begin{align*}
        \Vert \nun \ast \phi_{0, \sigmasq + \varsigmasqn} - \nuinfty \ast \phi_{0, \sigmasq + \varsigmasqn} \Verttwo
        \leq \Vert \nun - \nuinfty \Verttwo \Vert \phi_{0, \sigmasq + \varsigmasqn} \Vertone
        = \Vert \nun - \nuinfty \Verttwo = o(1).
    \end{align*}
    
    It is left to show that $\Vert \qn - \nun \ast \phi_{0, \sigmasq + \varsigmasqn} \Verttwo  = o(1)$.  
    Following the proof of Proposition \ref{proposition:fn}, let $\hn(\cdot | \thetan)$ and $\htilden(\cdot | \thetan)$ denote the distribution of $\sqrt{n} (\thetahatn - \thetan) | \thetan$ and $\sqrt{n} (\thetatilden - \thetan) | \thetan$ respectively.  Then, a few applications of change of variables yields
    \begin{align*}
        \qn(t) 
        &= \sqrt{n} \int_{-\infty}^{\infty} \hn(t - \sqrt{n} \theta | \theta) \pin(\theta) \lambda(d\theta) \\ 
        &= \sqrt{n} \int_{-\infty}^{\infty} \hn(t - \sqrt{n} \theta | \theta) \nun(\sqrt{n} \theta) \lambda(d\theta) \\ 
        &= \sqrt{n} \int_{-\infty}^{\infty} \nun(\sqrt{n} \theta) \lambda(d\theta) \int_{-\infty}^{\infty} \htilden(s | \theta) \phi(t - \sqrt{n} \theta - s | 0, \varsigmasqn) \lambda(ds)
    \end{align*}
    and 
    \begin{align*}
        (\nun \ast \phi_{0, \sigmasq + \varsigmasqn})(t) 
        &= \int_{-\infty}^{\infty} \nun(\theta) \phi(t - \theta | 0, \sigmasq + \varsigmasqn) \lambda(d\theta) \\ 
        &= \sqrt{n} \int_{-\infty}^{\infty} \nun(\sqrt{n} \theta) \phi(t - \sqrt{n}\theta | 0, \sigmasq + \varsigmasqn) \lambda(d\theta) \\ 
        &= \sqrt{n} \int_{-\infty}^{\infty} \nun(\sqrt{n} \theta) \lambda(d\theta) \int_{-\infty}^{\infty} \phi(s | 0, \sigmasq) \phi(t - \sqrt{n}\theta - s | 0, \varsigmasqn) \lambda(ds).
    \end{align*}
    Thus, combining the above terms, we have 
    \begin{align*}
        &\Vert \qn - \nun \ast \phi_{0, \sigmasq + \varsigmasqn)} \Vertinfty \\ 
        &\phantleq = \sup_{t \in \Rbb} \Big| \int_{-\infty}^{\infty} \sqrt{n} \nun(\sqrt{n} \theta) \lambda(d\theta) \int_{-\infty}^{\infty} [\gtilden(s | \theta) - \phi(s | 0, \sigmasq)]\phi(t - \sqrt{n} \theta - s | 0, \varsigmasqn) \lambda(ds) \Big| \\ 
        &\phantleq = \sup_{t \in \Rbb} \Big| \int_{-\infty}^{\infty} \sqrt{n} \nun(\sqrt{n} \theta) \lambda(d\theta) \int_{-\infty}^{\infty} [\Gtilden(s | \theta) - \Phi(s | 0, \sigmasq)]\phi'(t - \sqrt{n} \theta - s | 0, \varsigmasqn) \lambda(ds) \Big| \\ 
        &\phantleq \leq \sup_{t \in \Rbb} \int_{-\infty}^{\infty} \sqrt{n} \nun(\sqrt{n} \theta) \Big\{ \sup_{s \in \Rbb} \Big|\Gtilden(s|\theta) - \Phi(s|0, \sigmasq)\Big| \Big\} \lambda(d\theta) \int_{-\infty}^{\infty} \phi'(t - \sqrt{n}\theta - s|0, \varsigmasqn) \lambda(ds) \\ 
        &\phantleq = \sqrt{\frac{2}{\pi \varsigmasqn}} \int_{-\infty}^{\infty} \sqrt{n} \nun(\sqrt{n} \theta) \Big\{ \sup_{s \in \Rbb} \Big|\Gtilden(s|\theta) - \Phi(s|0, \sigmasq)\Big| \Big\} \lambda(d\theta) \\
        &\phantleq = \sqrt{\frac{2}{\pi \varsigmasqn}} \int_{-\infty}^{\infty} \nun(\theta) \Big\{ \sup_{s \in \Rbb} \Big|\Gtilden(s|\theta/\sqrt{n}) - \Phi(s|0, \sigmasq)\Big| \Big\} \lambda(d\theta).
    \end{align*}
    Since $\nun$ converges in total variation to $\nuinfty$, the sequence of measures is tight.  Hence, for $\epsilon > 0$, there exists a compact set $\Omega_{\epsilon}$ such that 
    \begin{align*}
        \int_{\Omega_{\epsilon}} \nun(\theta) \lambda(d\theta) < \epsilon / 4. 
    \end{align*}
    Moreover, equation \eqref{equation:uclt} implies that there exists an $N$ sufficiently large such that for all $n \geq N$ 
    \begin{align*}
        \sup_{\theta \in \Omega_{\epsilon}} \sup_{s \in \Rbb} \Big|\Gtilden(s|\theta/\sqrt{n}) - \Phi(s|0, \sigmasq)\Big| < \epsilon / 2.
    \end{align*}
    Thus, it follows that 
    \begin{align*}
        &\int_{-\infty}^{\infty} \nun(\theta) \Big\{ \sup_{s \in \Rbb} \Big|\Gtilden(s|\theta/\sqrt{n}) - \Phi(s|0, \sigmasq)\Big| \Big\} \lambda(d\theta) \\ 
        &\phantleq= \int_{\Omega_{\epsilon}} \nun(\theta) \Big\{ \sup_{s \in \Rbb} \Big|\Gtilden(s|\theta/\sqrt{n}) - \Phi(s|0, \sigmasq)\Big| \Big\} \lambda(d\theta)
        + \int_{\Omega_{\epsilon}^{\comp}} \nun(\theta) \Big\{ \sup_{s \in \Rbb} \Big|\Gtilden(s|\theta/\sqrt{n}) - \Phi(s|0, \sigmasq)\Big| \Big\} \lambda(d\theta) \\
        &\phantleq < (\epsilon / 2)\int_{\Omega_{\epsilon}} \nun(\theta) \lambda(d\theta) + 2 \int_{\Omega_{\epsilon}^{\comp}} \nun(\theta) \lambda(d\theta) \\ 
        &\phantleq < \epsilon.
    \end{align*}
    Since $\epsilon > 0$ is arbitrary, 
    \begin{align*}
        \int_{-\infty}^{\infty} \nun(\theta) \Big\{ \sup_{s \in \Rbb} \Big|\Gtilden(s|\theta/\sqrt{n}) - \Phi(s|0, \sigmasq)\Big| \Big\} \lambda(d\theta) = o(1).
    \end{align*}
    Therefore, there exists a sequence $\varsigmasqn \to 0$ such that 
    \begin{align*}
        \Vert \qn - \nun \ast \phi_{0, \sigmasq + \varsigmasqn} \Vertinfty = o(1).
    \end{align*}
    Invoking Lemma \ref{lemma:ltwo:lone} again finishes the proof.
\end{proof}

\begin{proof}[Proof of Theorem \ref{theorem:parametric}]
    In view of Lemma \ref{lemma:gaussian:tv}, it suffices to show that 
    \begin{align*}
        (\mupihat - \mupi)^{2} = o(\sigmapisq)
        \und
        \frac{\sigmapihatsq}{\sigmapisq} \conp 1.
    \end{align*}
    Indeed, we have
    \begin{align*}
        \mupihat - \mupi 
        &= \frac{1}{K} \sum_{k=1}^{K} \Big(\thetak + \frac{1}{m} \sum_{i=1}^{m} \psi(X_{i,k} ; \thetak) + \deltak\Big) - \mupi \\ 
        &= \frac{1}{K} \sum_{k=1}^{K} (\thetak - \mupi) + \frac{1}{mK} \sum_{k=1}^{K} \sum_{i=1}^{m} \psi(X_{i,k} ; \thetak) + \frac{1}{K} \sum_{k=1}^{K} \deltak \\ 
        &= \op(\sigmapi),
    \end{align*}
    where the last line is a consequence of Assumption \ref{assumption:peb:consistent}.
    Now, for $\sigmapihatsq$, it follows that
    \begin{align*}
        &\frac{1}{K} \sum_{k=1}^{K} (\thetahatk - \mupihat)^{2} - \frac{\sigmahatsq}{m} - \sigmapisq \\
        &\phantleq= \frac{1}{K} \sum_{k=1}^{K} (\thetak - \mupi)^{2}  - \sigmapisq + \frac{1}{K} \sum_{k=1}^{K} (\thetak - \thetahatk)^{2} - \frac{\sigmasq}{m} + \frac{2}{K} \sum_{k=1}^{K} (\thetak - \mupi)(\thetahatk - \thetak) - (\mupi - \mupihat)^{2} + \frac{\sigmasq - \sigmahatsq}{m}
    \end{align*}
    By properties of the Gaussian distribution, we note that 
    \begin{align*}
        \frac{1}{K} \sum_{k=1}^{K} (\thetak - \mupi)^{2} \sim \frac{\sigmapisq}{K} \chisq_{K-1}
    \end{align*}
    and so
    \begin{align*}
        \frac{1}{K} \sum_{k=1}^{K} (\thetak - \mupi)^{2} - \sigmapisq = \op(\sigmapisq).
    \end{align*}
    The other terms proceed similarly and are omitted.

    For coverage and length, the result follows from combining Propositions \ref{proposition:totalvariation} and \ref{proposition:qconvergence} and Corollary \ref{corollary:fn}.
\end{proof}

\begin{lemma}\label{lemma:kde}
    Suppose Assumptions \ref{assumption:berryesseen} and \ref{assumption:prior:tv} hold.
    If $\ansq\bnsq \to 0$, $\bnsq \to 0$, and $K \bn \to \infty$, then 
    \begin{align*}
        \Vert \pin - \pihatn \Vertone = \op(1).
    \end{align*}
\end{lemma}

\begin{proof}[Proof of Lemma \ref{lemma:kde}]
    Indeed, consider an auxiliary estimator
    \begin{align*}
        \qhatn(t) 
        = \frac{1}{K} \sum_{k=1}^{K} \phi(t | \an\thetahatk, \ansq\bnsq).
    \end{align*}
    Then, a change of variables immediately implies that 
    \begin{align*}
        \Vert \pin - \pihatn \Vertone 
        = \Vert \nun - \qhatn \Vertone 
        \leq \Vert \nun - \nuinfty \Vertone 
        + \Vert \nuinfty - \qhatn \Vertone.
    \end{align*}
    In view of Assumption \ref{assumption:prior:tv} and Lemma \ref{lemma:ltwo:lone}, it suffices to show that
    \begin{align*}
        \Ebb \Vert \nuinfty - \qhatn \Verttwo^{2} = o(1).
    \end{align*}
    Now, 
    \begin{align*}
        \Ebb \Vert \nuinfty - \qhatn \Verttwo^{2} 
        &= \Ebb \int_{-\infty}^{\infty} [\nuinfty(x) - \qhatn(x)]^{2} \lambda(dx) \\ 
        &= \int_{-\infty}^{\infty} \Ebb [\nuinfty(x) - \qhatn(x)]^{2} \lambda(dx) \\ 
        &= \int_{-\infty}^{\infty} \{ \Ebb [\nuinfty(x) - \qhatn(x)] \}^{2} \lambda(dx) 
        + \int_{-\infty}^{\infty} \var[\qhatn(x)] \lambda(dx).
    \end{align*}
    For the first term, we have 
    \begin{align*}
        \int_{-\infty}^{\infty} \{ \Ebb [\nuinfty(x) - \qhatn(x)] \}^{2} \lambda(dx) 
        &= \int_{-\infty}^{\infty} \lambda(dx) \Big\{\int_{-\infty}^{\infty} \Big[\nuinfty(x) - \phi(x | t, \ansq \bnsq)\Big] \qn(t) \lambda(dt) \Big\}^{2} \\ 
        &\leq 2 \int_{-\infty}^{\infty} \lambda(dx) \Big\{\int_{-\infty}^{\infty} \Big[\nuinfty(x) - \phi(x | t, \ansq \bnsq)\Big] \nuinfty(t) \lambda(dt) \Big\}^{2} \\ 
        &\phantleq + 2 \int_{-\infty}^{\infty} \lambda(dx) \Big\{\int_{-\infty}^{\infty} \phi(x | t, \ansq \bnsq) [\qn(t) - \nuinfty(t)] \lambda(dt) \Big\}^{2}.
    \end{align*}
    Since $\ansq \bnsq \to 0$, then 
    \begin{align*}
        \int_{-\infty}^{\infty} \lambda(dx) \Big\{\int_{-\infty}^{\infty} \Big[\nuinfty(x) - \phi(x | t, \ansq \bnsq)\Big] \nuinfty(t) \lambda(dt) \Big\}^{2}
        &= \int_{-\infty}^{\infty} [\nuinfty(x) - (\nuinfty \ast \phi_{0, \ansq \bnsq})(x)]^{2} \lambda(dx) \\ 
        &= \Vert \nuinfty - \nuinfty \ast \phi_{0, \ansq \bnsq} \Verttwo^{2} \\ 
        &= o(1)
    \end{align*}
    by Theorem 4.5.4 of \cite{bogachev2007measure}.  Moreover, by Jensen's inequality, it follows that 
    \begin{align*}
        \int_{-\infty}^{\infty} \lambda(dx) \Big\{\int_{-\infty}^{\infty} \phi(x | t, \ansq \bnsq) [\qn(t) - \nuinfty(t)] \lambda(dt) \Big\}^{2} 
        &\leq \int_{-\infty}^{\infty} \lambda(dx) \int_{-\infty}^{\infty} \phi(x | t, \ansq \bnsq) [\qn(t) - \nuinfty(t)]^{2} \lambda(dt)  \\ 
        &= \int_{-\infty}^{\infty} [\qn(t) - \nuinfty(t)]^{2} \lambda(dt) \int_{-\infty}^{\infty} \phi(x | t, \ansq \bnsq) \lambda(dx) \\ 
        &= \Vert \qn - \nuinfty \Verttwo^{2} \\ 
        &= o(1),
    \end{align*}
    where the limit is due to Proposition \ref{proposition:qconvergence}.
    Finally, note that 
    \begin{align*}
        \int_{-\infty}^{\infty} \var[\qhatn(x)] \lambda(dx) = \Ocal\Big( \frac{1}{K\bn} \Big) = o(1)
    \end{align*}
    by Proposition 1.4 of \cite{tsybakov2009introduction}.  Combining the above calculations finishes the proof.
\end{proof}

\begin{lemma}\label{lemma:deconvolution}
    Suppose Assumptions \ref{assumption:berryesseen}, \ref{assumption:prior:tv}, and \ref{assumption:fourier} hold.  There exist a sequence $\{\bn\}_{n=1}^{\infty}$ with $\bn\sqrt{n} \to 0$ such that
    \begin{align*}
        \Vert \pitilden - \pin \Vertone = \op(1).
    \end{align*}
\end{lemma}

\let\pihatnold\pihatn
\let\nuhatnold\nuhatn
\let\pihatn\pitilden
\let\nuhatn\nutilden

\begin{proof}[Proof of Lemma \ref{lemma:deconvolution}]
    Again, let 
    \begin{align*}
        \nuhatn(x) \defined \frac{1}{\sqrt{n}} \pihatn(x / \sqrt{n}) 
        &= \frac{1}{2\pi K \sqrt{n}} \sum_{k=1}^{K} \int_{-1/\bn}^{1/\bn} \exp( -\ifrak x z / \sqrt{n}) \exp( \ifrak z \thetahatk) \exp \Big( \frac{\sigmasq z^{2}}{2n} \Big) \lambda(dz) \\ 
        &= \frac{1}{2\pi K} \sum_{k=1}^{K} \int_{-1/(\bn\sqrt{n})}^{1/(\bn\sqrt{n})} \exp( -\ifrak xz) \exp( \ifrak z \sqrt{n}\thetahatk) \exp \Big( \frac{\sigmasq z^{2}}{2} \Big) \lambda(dz).
    \end{align*}
    Then, 
    \begin{align*}
        \Vert \pin - \pihatn \Vertone 
        = \Vert \nun - \nuhatn \Vertone 
        \leq \Vert \nun - \nuinfty \Vertone 
        + \Vert \nuinfty - \nuhatn \Vertone.
    \end{align*}
    We again show that 
    \begin{align*}
        \Ebb \Vert \nuinfty - \nuhatn \Verttwo^{2} = o(1).
    \end{align*}
    By Parseval's theorem, it follows that 
    \begin{align*}
        \Ebb \Vert \nuinfty - \nuhatn \Verttwo^{2} = \Ebb \Vert \nuftinfty - \nuhatftn \Verttwo^{2},
    \end{align*}
    where $\nuftinfty$ and $\nuhatftn$ are the corresponding Fourier transforms of $\nuinfty$ and $\nuhatn$.  
    However, by the Fourier inversion theorem, we have
    \begin{align*}
        \nuhatftn(x) = \indic{x \in (-1/(\bn\sqrt{n}), 1/(\bn\sqrt{n}))} \frac{1}{K} \sum_{k=1}^{K} \exp(\ifrak x \sqrt{n} \thetahatk) \exp\Big( \frac{\sigmasq x^{2}}{2} \Big).
    \end{align*}
    Recall that 
    \begin{align*}
        \Ebb \Vert \nuftinfty - \nuhatftn \Verttwo^{2}
        = \int_{-\infty}^{\infty} \{ \Ebb [\nuftinfty(x) - \nuhatftn(x)] \}^{2} \lambda(dx) 
        + \int_{-\infty}^{\infty} \var[\nuhatftn(x)] \lambda(dx).
    \end{align*}
    Now, 
    \begin{align*}
        &\int_{-\infty}^{\infty} \Ebb[\nuftinfty(x) - \nuhatftn(x)]^{2} \lambda(dx) \\ 
        &\phantleq= \int_{-\infty}^{\infty} \lambda(dx) \Big\{ \int_{-\infty}^{\infty} \Big[ \nuftinfty(x) - \indic{x \in (-1/(\bn\sqrt{n}), 1/(\bn\sqrt{n}))} \exp(\ifrak x t) \exp\Big( \frac{\sigmasq x^{2}}{2} \Big) \Big] \qn(t) \lambda(dt) \Big\}^{2} \\ 
        &\phantleq\leq 2\int_{-\infty}^{\infty} \lambda(dx) \Big\{ \int_{-\infty}^{\infty} \Big[ \nuftinfty(x) - \indic{x \in (-1/(\bn\sqrt{n}), 1/(\bn\sqrt{n}))} \exp(\ifrak x t) \exp\Big( \frac{\sigmasq x^{2}}{2} \Big) \lambda(dz) \Big] (\nuinfty\ast\phi_{0, \sigmasq})(t) \lambda(dt) \Big\}^{2} \\ 
        &\phantleq\phantleq +2\int_{-\infty}^{\infty} \lambda(dx) \Big\{ \int_{-\infty}^{\infty} \indic{x \in (-1/(\bn\sqrt{n}), 1/(\bn\sqrt{n}))} \exp(\ifrak x t) \exp\Big( \frac{\sigmasq x^{2}}{2} \Big) [\qn(t) - (\nuinfty\ast\phi_{0, \sigmasq})(t)] \lambda(dt)  \Big\}^{2}.
    \end{align*}
    For the first term, note that 
    \begin{align*}
        &\int_{-\infty}^{\infty} \lambda(dx) \Big\{ \int_{-\infty}^{\infty} \Big[ \nuftinfty(x) - \indic{x \in (-1/(\bn\sqrt{n}), 1/(\bn\sqrt{n}))} \exp(\ifrak x t) \exp\Big( \frac{\sigmasq x^{2}}{2} \Big) \lambda(dz) \Big] (\nuinfty\ast\phi_{0, \sigmasq})(t) \lambda(dt) \Big\}^{2} \\ 
        &\phantleq = \int_{x \not\in (-1/(\bn\sqrt{n}), 1/(\bn\sqrt{n}))} [\nuftinfty(x)]^{2} \lambda(dx) \\ 
        &\phantleq = o(1)
    \end{align*}
    as $1 / (\bn\sqrt{n}) \to \infty$ by the dominated convergence theorem.
    
    For the other term, we see that 
    \begin{align*}
        &\int_{-\infty}^{\infty} \lambda(dx) \Big\{ \int_{-\infty}^{\infty} \indic{x \in (-1/(\bn\sqrt{n}), 1/(\bn\sqrt{n}))} \exp(\ifrak x t) \exp\Big( \frac{\sigmasq x^{2}}{2} \Big) [\qn(t) - (\nuinfty\ast\phi_{0, \sigmasq})(t)] \lambda(dt)  \Big\}^{2} \\ 
        &\phantleq = \int_{-1 / (\bn\sqrt{n})}^{1 / (\bn\sqrt{n})} \exp(\sigmasq x^{2}) [\qftn(x) - \nuftinfty(x)\phift_{0, \sigmasq}(x)]^{2} \lambda(dx) \\ 
        &\phantleq \leq \exp\Big(\frac{\sigmasq}{\bnsq n} \Big) \Vert \qftn - \nuftinfty \phift_{0, \sigmasq} \Verttwo^{2} \\ 
        &\phantleq \leq \exp\Big(\frac{\sigmasq}{\bnsq n} \Big) \Vert \qn - \nuinfty \ast \phi_{0, \sigmasq} \Verttwo^{2}.
    \end{align*}
    Since $\Vert \qn - \nuinfty \ast \phi_{0, \sigmasq} \Verttwo^{2} = o(1)$ by Proposition \ref{proposition:qconvergence}, there exists a sequence $\bn \to 0$ with $\bn \sqrt{n} \to 0$ such that the above is $o(1)$.  
    
    Finally, applying Proposition 2.1 of \cite{meister2009deconvolution} shows that 
    \begin{align*}
        \int_{-\infty}^{\infty} \var[\nuhatftn(x)] \lambda(dx) = o(1)
    \end{align*}
    and combining the above calculations finishes the proof.
\end{proof}

\let\pihatn\pihatnold
\let\nuhatn\nuhatnold

\begin{lemma}\label{lemma:deconvolution:two}
    Consider the setting of Lemma \ref{lemma:deconvolution}.  If $\{\kappan\}_{n=1}^{\infty}$ are convex weights with $\kappan \to 1$, then
    \begin{align*}
        \Vert \pihatn - \pin \Vertone = \op(1).
    \end{align*}
\end{lemma}

\begin{proof}[Proof of Lemma \ref{lemma:deconvolution:two}]
    Indeed, we have 
    \begin{align*}
        \Vert \pihatn - \pin \Vertone 
        &= \Vert \kappan \max\{\pitilden, 0\} + (1 - \kappan) \phi_{\mupihat, \sigmapihatsq} - \pin \Vertone \\ 
        &\leq \Vert \max\{ \pitilden, 0\} - \pin \Vertone + (1 - \kappan) \Big[\Vert \max\{\pitilden, 0\} \Vertone + \Vert \phi_{\mupihat, \sigmapihatsq} \Vertone \Big].
    \end{align*}
    Since $\pin \geq 0$, it follows that 
    \begin{align*}
        \Vert \max\{ \pitilden, 0\} - \pin \Vertone \leq \Vert \pitilden - \pin \Vertone = \op(1)
    \end{align*}
    by Lemma \ref{lemma:deconvolution}.  Moreover, $\Vert \max\{ \pitilden, 0 \} \Vertone + \Vert \phi_{\mupihat, \sigmapihatsq} \Vertone \leq 2$, implying 
    \begin{align*}
        (1 - \kappan) \Big[\Vert \max\{\pitilden, 0\} \Vertone + \Vert \phi_{\mupihat, \sigmapihatsq} \Vertone \Big] = \op(1).
    \end{align*}
    This finishes the proof.
\end{proof}

\begin{proof}[Proof of Theorem \ref{theorem:np}]
    The proof follows immediately by combining Lemmata \ref{lemma:kde} and \ref{lemma:deconvolution:two} with Proposition \ref{proposition:totalvariation} and Corollary \ref{corollary:fn}.
\end{proof} 

\bibliographystyle{newapa}
\bibliography{DR}

\end{document}